\documentclass[10pt]{amsart}

\usepackage{amssymb,amsmath,amsthm,amsfonts,microtype}
\usepackage{pdfsync}

\newcommand{\comment}[1]{}

\newtheorem{lem}{Lemma}[section]
\newtheorem{propn}{Proposition}[section]
\newtheorem{cor}{Corollary}[section]
\newtheorem{thm}{Theorem}[section]

\theoremstyle{remark}

\theoremstyle{definition}
\newtheorem{defn}{Definition}[section]

\newcommand{\R}{\mathbb R}

\newcommand{\NN}{\mathcal N}

\newcommand{\EE}{\mathcal E}
\newcommand{\UU}{\mathcal U}

\newcommand{\PP}{\mathcal P}
\newcommand{\QQ}{\mathcal Q}
\newcommand{\RR}{\mathcal R}
\newcommand{\mS}{\mathcal S}

\newcommand{\D}{\delta}
\newcommand{\de}{\delta}
\newcommand{\eps}{\varepsilon}

\newcommand{\VE}{\varepsilon}
\newcommand{\A}{\alpha}
\newcommand{\al}{\alpha}
\newcommand{\B}{\beta}
\newcommand{\Be}{\beta}
\newcommand{\ga}{\gamma}
\newcommand{\lm}{\lambda}

\newcommand{\be}{\begin{equation}}
\newcommand{\ee}{\end{equation}}
\newcommand{\bee}{\begin{equation*}}
\newcommand{\eee}{\end{equation*}}

\begin{document}
\title{Product of Simplices and sets of positive upper density in $\R^d$}
%\title{Rectangles and Product of Simplices \\ in sets of positive upper density in $\R^d$}
%\title{Rectangles and sets of positive upper density in $\R^d$}
%Embedding squares into sets of positive upper density in $\R^4$}
%{Squares are Density Ramsey}
\author{Neil Lyall\quad\quad\quad\'Akos Magyar}
\thanks{The first and second authors were partially supported by Simons Foundation Collaboration Grant for Mathematicians 245792 and by Grants NSF-DMS 1600840 and  ERC-AdG 321104, respectively.}

\address{Department of Mathematics, The University of Georgia, Athens, GA 30602, USA}
\email{lyall@math.uga.edu}
\address{Department of Mathematics, The University of Georgia, Athens, GA 30602, USA}
\email{magyar@math.uga.edu}

%\keywords{}

\begin{abstract}
We establish that any subset of $\mathbb{R}^d$ of positive upper Banach density necessarily contains an isometric copy of all sufficiently large dilates of any fixed two-dimensional rectangle provided $d\geq4$.

We further present an extension of this result to configurations that are the product of two non-degenerate simplices; specifically we show that if $\Delta_{k_1}$ and $\Delta_{k_2}$ are two fixed non-degenerate simplices of $k_1+1$ and $k_2+1$ points respectively, then
any subset of $\mathbb{R}^d$ of positive upper Banach density with $d\geq k_1+k_2+6$  will necessarily contain an isometric copy of all sufficiently large dilates of $\Delta_{k_1}\times\Delta_{k_2}$.

A new direct proof of the fact that any subset of $\mathbb{R}^d$ of positive upper Banach density necessarily contains an isometric copy of all sufficiently large dilates of any fixed non-degenerate simplex of $k+1$ points provided $d\geq k+1$, a result originally due to Bourgain, is also presented.
\end{abstract}
%\date{\today}
\maketitle
%\tableofcontents

\setlength{\parskip}{1pt}
%\setlength{\parindent}{0pt}

%%%%%%%%%%%%%%%%%%%%%%%%%%%%%%%%%%%%%%%%%%%%%%%

\section{Introduction}\label{intro}

\subsection{Background}

Recall that the \emph{upper Banach density} of a measurable set $A\subseteq\R^d$ is defined by
\be\label{BD}
\D^*(A)=\lim_{N\rightarrow\infty}\sup_{t\in\R^d}\frac{|A\cap(t+Q_N)|}{|Q_N|},\ee
where $|\cdot|$ denotes Lebesgue measure on $\R^d$ and $Q_N$ denotes the cube $[-N/2,N/2]^d$.

A result of Katznelson and Weiss \cite{FKW} states that
if $A\subseteq\mathbb{R}^2$ has positive upper Banach density, then its distance set
\[\text{dist}(A)=\{|x-x'|\,:\, x,x'\in A\}\] contains all large numbers.
This result was later reproved using Fourier analytic techniques by Bourgain in \cite{B} where he established the following more general result for arbitrary non-degenerate $k$-dimensional simplices.

\begin{thm}[Bourgain \cite{B}]\label{BourSimp}
Let $\Delta_k\subseteq\R^{k}$ be a fixed non-degenerate $k$-dimensional simplex.

If $A\subseteq\R^d$ has positive upper Banach density and $d\geq k+1$, then
 there exists a threshold $\lm_0=\lm_0(A,\Delta_k)$ such that $A$ contains an isometric copy of $\lm\cdot\Delta_k$ for all $\lambda\geq \lm_0$.
\end{thm}

Recall that a set $\Delta_k=\{0,v_1,\dots,v_k\}$ of $k+1$ points in $\R^{k}$ is a non-degenerate $k$-dimensional simplex if the vectors $v_1,\dots,v_k$ are linearly independent and
 that a configuration $\Delta_k'$ is  an isometric copy of $\lm\cdot \Delta_k$ in $\R^d$ if $\Delta'_k=x+\lm\cdot U(\Delta_k)$ for
some $x\in \R^d$ and $U\in SO(d)$ when $d\geq k+1$.

\subsection{Main Results}

In Section \ref{newdistances} we present a new and direct proof of Theorem \ref{BourSimp} when $k=1$, namely a new proof of the aforementioned distance set result of Katznelson and Weiss. A new direct proof of Theorem \ref{BourSimp} in its full generality is also given, in fact two different new approaches are presented in Section \ref{newsimplices}.
However, the main purpose of this article is to establish the following new results, namely Theorems \ref{Rect} and \ref{ProdSimp} below.

\begin{thm}\label{Rect}
Let $\Box=\{0,v_1,v_2, v_1+v_2\}\subseteq\R^2$ with $v_1\cdot v_2=0$ denote a fixed two-dimensional rectangle.

If $A\subseteq\R^d$ has positive upper Banach density and $d\geq 4$, then
 there exists a threshold $\lm_0=\lm_0(A,\Box)$ such that $A$ contains an isometric copy of $\lm\cdot\Box$ for all $\lambda\geq \lm_0$.
\end{thm}

%In Theorem \ref{Rect}, and its proof, we may assume that $v_1,v_2\in\R$ and \[\Box=\{0,v_1\}\times\{0,v_2\}=\{(x,y)\in\R\times\R\,:\,x\in\{0,v_1\},\,y\in\{0,v_2\}\}.\] 
Since $d\geq4$ we can write $\R^d=\R^{d_1}\times\R^{d_2}$ with $d_1,d_2\geq2$. It is important to note that the isometric copies of $\lm\cdot\Box$, whose existence in $A$ Theorem \ref{Rect} guarantees, will in fact all be of the  special form
\[\{(x,y), (x',y), (x,y'), (x',y')\}\subseteq \R^{d_1}\times\R^{d_2}\]
where $|x-x'|=\lm|v_1|$ and $|y-y'|=\lm|v_2|.$

%$\Delta_1'\times\Delta_2'$ with each $\Delta_i'\subseteq\R^{d_i}$ with $d_i\geq2$ an  isometric copy of $\lm\cdot\Delta_i$.

We also establish the following generalization of Theorem \ref{Rect}, but with a slight loss in the dimension $d$.

\begin{thm}\label{ProdSimp}
Let $\Delta_{k_1}$ and $\Delta_{k_2}$ be two fixed non-degenerate simplices of dimension $k_1$ and $k_2$.

If $A\subseteq\R^d$ has positive upper Banach density with $d\geq k_1+k_2+6$, then
 there exists a threshold $\lm_0=\lm_0(A,\Delta_{k_1},\Delta_{k_2})$ such that $A$ contains an isometric copy of $\lm\cdot(\Delta_{k_1}\times\Delta_{k_2})$  of the form $\Delta_{k_1}'\times\Delta_{k_2}'$ with each $\Delta_{k_i}'\subseteq\R^{d_i}$ an  isometric copy of $\lm\cdot\Delta_{k_i}$ for all $\lambda\geq \lm_0$.
\end{thm}
% \be \Delta_1\times\Delta_2=\{(x,y)\in\R^{k_1}\times\R^{k_2}\,:\,x\in\Delta_1,\,y\in\Delta_2\}. \ee

It will be clear from the proofs of Theorems \ref{ProdSimp} and \ref{Rect} that if $1=k_1<k_2$, then the conclusion of Theorem \ref{ProdSimp} will in fact hold under the weaker hypothesis that $d\geq k_1+k_2+4$.

Note further that if $A$ were a direct product set $B_1\times B_2\subseteq \R^{d_1}\times \R^{d_2}$ with each $d_i\geq k_i+1$, then the conclusion of Theorem \ref{ProdSimp}  (which contains the conclusion of Theorem \ref{Rect}  when each $k_i=1$) would follow immediately from  Theorem \ref{BourSimp} and under the
 weaker hypothesis that $d\geq k_1+k_2+2$.

%Key to our proofs of Theorems \ref{Rect} and  \ref{ProdSimp} is the fact that we shall restrict our attention to those isometric copies of $\lm\cdot(\Delta_1\times \Delta_2)$ that are of the form $\Delta_1'\times\Delta_2'$ with each $\Delta_i'\subseteq\R^{d_i}$ an  isometric copy of $\lm\cdot\Delta_i$.

%is that we shall restrict our attention to those isometric copies of $\lm\cdot\Box$ that are of the form $\Delta_1'\times\Delta_2'$ with each $\Delta_i'\subseteq\R^{d_i}$ an isometric copy of the two point configuration $\{0,\lm\cdot v_i\}$.

The natural extension of Theorems \ref{Rect} and  \ref{ProdSimp} to $\ell$-dimensional rectangles and $\ell$-fold products of simplices (with $\ell>2$) also holds, but as the arguments involved in establishing these results are significantly more technical than those needed for Theorems \ref{Rect} and \ref{ProdSimp} we plan to address this in a separate article.

\subsection{Outline of Paper}

Our approach to proving Theorems \ref{Rect} and \ref{ProdSimp} will be to reduce them to quantitative results in the compact setting of $[0,1]^{d_1}\times[0,1]^{d_2}$, namely Propositions \ref{Propn1} and \ref{Propn11}. These reductions are carried out in Section \ref{red} with the remainder of Section \ref{4} and the entirety of
Sections \ref{SecPart1}-\ref{SecPart2} then devoted to establishing Propositions \ref{Propn1} and \ref{Propn11}.

In Section \ref{newdistances} we present a new direct proof of Theorem \ref{BourSimp} when $k=1$ and two new proofs of  Theorem \ref{BourSimp}, in its full generality, are  presented in Section \ref{newsimplices}. In both cases our novel approach will be to first reduce matters to results for suitably uniformly distributed subsets of $[0,1]^d$.

%\section{Outline of Paper}
\comment{
Our approach to proving both Theorem \ref{BourSimp} and Theorem \ref{ProdSimp} will be to reduce them to results in the compact setting of $[0,1]^d$ and $[0,1]^{d_1}\times[0,1]^{d_2}$ respectively, namely Proposition \ref{Propn00} and Proposition \ref{Propn1}. This reduction is carried out in Section \ref{red} below,
the proof of Proposition \ref{Propn00} is carried out in Sections \ref{new2}, and the remainder of the article, namely Sections \ref{SecPart1}-\ref{reglemproof}, is devoted to establishing Proposition \ref{Propn11}.
}

%We further remark that the proof of this \emph{distance set} result is especially simple if one assumes that the sets $B_i$ are \emph{uniformly distributed}, in a precise sense that we define below, and that it is this observation that leads to a new direct proof of Theorem \ref{BourSimp}, see Section \ref{NewProof}.

%In order to prove Theorem \ref{Thm1} it suffices to establish the following quantitative result for subsets of $[0,1]^4$.
%\subsection{Outline of Paper}

%\section{Proof of Theorem \ref{ProdSimp}}

%\newpage
\section{Uniformly Distributed Subsets of $\R^d$ and a New Proof of Theorem \ref{BourSimp} when $k=1$}\label{newdistances}
%\section{Distances in uniformly distributed subsets of $\R^d$}\label{red}
%\section{Results in $[0,1]^d$ and $[0,1]^{d_1}\times[0,1]^{d_2}$ that imply Theorems \ref{BourSimp} and \ref{ProdSimp}}\label{red}

%\subsection{Reducing Theorem \ref{BourSimp} to a result for %simplices in uniformly distributed subsets of $[0,1]^d$}\label{new1}
In this section we introduce a precise notion of uniform distribution for subsets of $\R^d$ and prove an (optimal) result, Proposition \ref{Propn0} below, on distances in uniformly distributed subsets of $[0,1]^d$. Proposition \ref{Propn0} will be critically important in our proof of Proposition \ref{Propn1}, but as we shall see below it also immediately implies Theorem \ref{BourSimp} when $k=1$ and hence provides a new direct proof of the following

\begin{thm}[Katznelson and Weiss \cite{FKW}]\label{KW}
If $A\subseteq\R^d$ has positive upper Banach density and $d\geq 2$, then
 there exists a threshold $\lm_0=\lm_0(A)$  such that for all $\lambda\geq \lm_0$ there exist a pair of points
\bee\{x,x'\}\subseteq A\quad\text{with}\quad |x-x'|=\lm.\eee
\end{thm}

\subsection{Uniform Distribution and Distances}

\begin{defn}[$(\VE,L)$-uniform distribution]
Let $0<L\leq\VE\ll1$ and $Q_L=[-L/2,L/2]^d$.

A set $A\subseteq[0,1]^d$ is said to be $(\VE,L)$-uniformly distributed if
\be\label{3.3.1}
\int_{[0,1]^d}\left|\frac{|A\cap(t+Q_{L})|}{|Q_{L}|}-|A|\,\right|^2\,dt\leq\VE^2.
\ee
\end{defn}
%Note that if a set is $(\VE,L)$-uniformly distributed then it is also $(\VE,L')$-uniformly distributed for any $L'\geq L$.

\begin{propn}[Distances in uniformly distributed sets]\label{Propn0}
Let $0<c\leq1$, $0<\lm\leq\VE\ll 1$ and $d\geq2$.

If $A\subseteq[0,1]^d$
is $(\VE,\VE^4\lm)$-uniformly distributed  with $\A=|A|>0$% in the sense that $\|f_A\|_{U^1(\VE^4\lm)}\leq\VE$
, then there exist a pair of points
%$\{x_1,x_2\}\subseteq A$ with $|x_2-x_1|=\lm$, in fact
\bee\{x,x'\}\subseteq A\quad\text{with}\quad |x-x'|=c\lm.\eee

In fact,
\bee
%\frac{1}{2\pi\lm}
\iint1_A(x)1_A(x-c\lm x_1)\,d\sigma(x_1)\,dx=\A^2 +O(c^{-1/6}\VE^{2/3}).
\eee
%{\substack{x_1,x_2\in [0,1]^2\\|x_2-x_1|=\lm}}1_A(x_1)1_A(x_2)\,dx_1\,dx_2=|A|^2 +O(\VE).
where $\sigma$ denotes the normalized measure on the sphere $\{x\in\R^d\,:\,|x|=1\}$ induced by Lebesgue measure.
%denotes normalized arc-length measure on the circle $\{x'\in\R^2\,:\,|x'|=1\}$.
%where $d\sigma$ denotes normalized arc-length measure on the unit circle in $\R^2$.
\end{propn}

Before proving Proposition \ref{Propn0} we will first show that when $c=1$ it immediately implies Theorem \ref{KW}. To the best of our knowledge this observation, which gives a direct proof of Theorem \ref{KW}, is new.

\subsection{Proof that Proposition \ref{Propn0} implies Theorem \ref{KW}}\label{newreduction}
Let $\VE>0$ and $A\subseteq\R^d$  with $\D^*(A)>0$.\smallskip
\comment{Recall that the \emph{upper Banach density} $\D^*$ of a measurable set $A\subseteq\R^d$ is defined by
\be\label{BD}
\D^*(A)=\lim_{N\rightarrow\infty}\sup_{x\in\R^d}\frac{|A\cap(x+Q_N)|}{|Q_N|},\ee
where $|\cdot|$ denotes Lebesgue measure on $\R^d$ and $Q_N$ denotes the cube $[-N/2,N/2]^d$.}
%\subsection{Proposition \ref{Propn0} implies Theorem \ref{Thm0}}\label{Propn0 implies Thm0}
%\[\D^*(A):=\lim_{N\rightarrow\infty}\sup_{x\in\R^d}\frac{|A\cap(x+Q_N)|}{|Q_N|}>0.\]
%where $|\cdot|$ denotes Lebesgue measure on $\R^2$ and $Q_N$ the cube $[-N/2,N/2]^2$.

 The following two facts follow immediately from the definition of upper Banach density, see (\ref{BD}):
\begin{itemize}
\item[(i)] There exist $M_0=M_0(A,\VE)$ such that for all $M\geq M_0$ and all $t\in\R^d$
\[\frac{|A\cap(t+Q_M)|}{|Q_M|}\leq(1+\VE^4/3)\,\D^*(A).\]
\item[(ii)] There exist arbitrarily large $N\in\R$ such that
\[\frac{|A\cap(t_0+Q_N)|}{|Q_N|}\geq(1-\VE^4/3)\,\D^*(A)\]
for some $t_0\in\R^d$.
\end{itemize}

Combining (i) and (ii) above we see that for any $\lm\geq \VE^{-4}M_0$, there exist $N\geq\VE^{-4}\lm$ and $t_0\in\R^d$ such that
\bee
\frac{|A\cap(t+Q_{\VE^4\lm})|}{|Q_{\VE^4\lm}|}\leq(1+\VE^4)\frac{|A\cap(t_0+Q_N)|}{|Q_N|}
\eee
for all $t\in\R^d$.
%Theorem \ref{BourSimp} thus reduces, via a rescaling of $A\cap(t_0+Q_N)$ to a subset of $[0,1]^d$, to Proposition \ref{Propn00} below.
Consequently, Theorem \ref{KW} reduces, via a rescaling of $A\cap(t_0+Q_N)$ to a subset of $[0,1]^d$,
to establishing that if
%\begin{propn}\label{Propn00}
%Let
$0<\lm\leq\VE\ll1$ and $A\subseteq[0,1]^d$ is measurable with $|A|>0$ and the property  that
\bee
\frac{|A\cap(t+Q_{\VE^4\lm})|}{|Q_{\VE^4\lm}|}\leq (1+\VE^4)\,|A|
\eee
for all $t\in\R^d$,  then there exist a pair of points $x,x'\in A$ such that
$
|x-x'|=\lm$.
Now since $A\cap(t+Q_{\VE^4\lm})$ is only supported in $[-\VE^4\lm,1+\VE^4\lm]^d$ it follows that
\be\label{Q-QL}
|A|=\int_{\R^d}\frac{|A\cap(t+Q_{\VE^4\lm})|}{|Q_{\VE^4\lm}|}\,dt=\int_{[0,1]^d}\frac{|A\cap(t+Q_{\VE^4\lm})|}{|Q_{\VE^4\lm}|}\,dt+O(\VE^4|A|),
\ee
from which one can easily deduce that
\be\label{most t's}
\Bigl|\Bigl\{t\in[0,1]^d\,:\, \frac{|A\cap(t+Q_{\VE^4\lm})|}{|Q_{\VE^4\lm}|}\leq (1-\VE^2)\,|A|\Bigr\}\Bigr|=O(\VE^2)
\ee
and hence that $A$ is $(\VE,\VE^4\lm)$-uniformly distributed. The result therefore follows, provided $d\geq 2$. \qed
%It follows from (\ref{most t's}) that
%\be
%\int_{[0,1]^d}\left|\frac{|A\cap(t+Q_{\VE^4\lm})|}{|Q_{\VE^4\lm}|}-|A|\,\right|^2\,dt=O(\VE^2).\qedhere
%\ee
%and hence, using (\ref{relate}), that $\|f_A\|_{U^1(\VE^4\lm)}=O(\VE)$.
%\end{proof}

\subsection{Proof of Proposition \ref{Propn0}}
\begin{defn}[Counting Function for Distances]

For $0<\lm\ll1$ and functions
\[f_0,f_{1}:[0,1]^{d}\to\R\] with $d\geq 2$ we define
\be
T(f_0,f_1)(\lm)= \iint f_0(x)f_1(x-\lm x_1)\,d\sigma(x_1) \,dx.
%\idotsint\limits_{x_1,\dots,x_k\in[0,1]^d} f_1(x_1)\cdots f_k(x_k)   \,\Omega^{(k-1)}_{\lm}(x_1,\dots,x_k)\,dx_1\cdots\,dx_k
\ee
%\be
%\Omega^{(k-1)}_{\lm}(x_1,\dots,x_k)=\prod_{j=2}^{k}\sigma_{\lm}^{x_1,\dots,x_{j-1}}(x_j-x_1).
%\ee
\end{defn}

\begin{defn}[$U^1(L)$-norm]
For $0<L\ll1$ and functions $f:[0,1]^d\to\R$ we define
\be
\|f\|_{U^1(L)}^2=\int\limits_{[0,1]^d}\Bigl|\frac{1}{L^d}\int\limits_{t+Q_L}f(x)\,dx\Bigr|^2 \,dt
=\int\limits_{[0,1]^d}\biggl(\frac{1}{L^{2d}}\iint\limits_{x,x'\in t+Q_L}f(x)f(x')\,dx'\,dx\biggr)dt
\ee
where $Q_L=[-L/2,L/2]^d$.
\end{defn}

It is  an easy, but important, observation that
\be\label{almostU1}
\|f\|_{U^1(L)}^2=\iint f(x)f(x-x_1)\psi_L(x_1)\,dx_1\,dx +O(L),
\ee
where
%\be
$\psi_L=L^{-2d}\,1_{Q_L}*1_{Q_L}$.
%\ee
Note also that if  $A\subseteq[0,1]^d$ with $\A=|A|>0$ and we define
\bee
f_A:=1_A-\A1_{[0,1]^d}
\eee
then
\be\label{relate}
\int\limits_{[0,1]^d}\Bigl|\frac{1}{L^d}\int\limits_{t+Q_L}f_A(x)\,dx\Bigr|^2 \,dt=\int\limits_{[0,1]^d}\left|\frac{|A\cap(t+Q_{L})|}{|Q_{L}|}-|A|\,\right|^2\,dt+O(L).
\ee

Evidently the $U^1(L)$-norm is measuring the mean-square uniform distribution of $A$ on scale $L$. Specifically if $A$ is $(\VE,L)$-uniformly distributed, then $\|f_A\|_{U^1(L)}\leq2\VE$  provided $0<L\ll\VE$.

\smallskip

At the heart of this short proof of Proposition \ref{Propn0} is the following ``generalized von-Neumann inequality''.

\begin{lem}[Generalized von-Neumann for Distances]\label{GvN0}
For any $c>0$,
 $0<\VE,\lm\ll\min\{1,c^{-1}\}$ and functions
\[f_0,f_1:[0,1]^{d}\to[-1,1]\]
with $d\geq2$ we have
\bee
\left|T(f_0,f_1)(c\lm)\right|\leq\prod_{j=0,1}\|f_j\|_{U^1(\VE^4\lm)}+O(c^{-1/6}\VE^{2/3}).
\eee
\end{lem}

Indeed, if  $A\subseteq[0,1]^d$ with $d\geq 2$ and $\A=|A|>0$, then Lemma \ref{GvN0} implies that
\[\left|T(1_A,1_A)(c\lm)-T(\A1_{[0,1]^d},\A1_{[0,1]^d})(c\lm)\right|\leq 3\,\|f_A\|_{U^1(\VE^4\lm)}+O(c^{-1/6}\VE^{2/3})\]
for any $0<c\leq1$ and $0<\VE,\lm \ll1$.  Since
$T(\A1_{[0,1]^d},\A1_{[0,1]^d})(c\lm)=\A^{2}+O(c\lm)$ it follows that
\bee
T(1_A,1_A)(c\lm)=\A^{2}+O(c^{-1/6}\VE^{2/3})\eee
provided $0<\lm\leq\VE\ll1$.
\\

To finish the proof of Proposition \ref{Propn0} we are therefore left with the task of proving Lemma \ref{GvN0}.

\begin{proof}[Proof of Lemma \ref{GvN0}]
An application of Parseval followed by Cauchy-Schwarz implies that
\begin{align*}
T(f_{0},f_{1})(c\lm)^2&=\Bigl(\iint f_0(x)f_1(x-c\lm x_1)\,d\sigma(x_1) \,dx\Bigr)^2\\
&\leq \Bigl(\int_{\R^d} |\widehat{f_0}(\xi)||\widehat{f_1}(\xi)||\widehat{\sigma}(c\lm\xi)|\,d\xi\Bigr)^2\\
&\leq \prod_{j=0,1}\int_{\R^d} |\widehat{f_j}(\xi)|^2|\widehat{\sigma}(c\lm\xi)|\,d\xi
%&=\prod_{j=1,2}\left[\int_{\R^2} |\widehat{f_j}(\xi)|^2|\widehat{\sigma}(\lm\xi)|\widehat{\psi}(\VE\lm\xi)\,d\xi+\int_{\R^2} |\widehat{f_j}(\xi)|^2|\widehat{\sigma}(\lm\xi)|(1-\widehat{\psi}(\VE\lm\xi))\,d\xi\right]\\
%&\leq \prod_{j=1,2}\int_{\R^d} |\widehat{f_j}(\xi)|^2\widehat{\psi}(\VE\lm\xi)\,d\xi+O(\VE^{1/3}),
\end{align*}
where
\bee
\widehat{\mu}(\xi)=\int_{\R^d}e^{-2\pi i x\cdot\xi}\,d\mu(x)\eee
denotes the Fourier transform of any complex-valued Borel measure $d\mu$ and $\widehat{g}(\xi)$ is the Fourier transform of the measure $d\mu=g\,dx$. Combining the basic fact (see for example \cite{Stein}) that
\bee
|\widehat{\sigma}(\xi)|\leq\min\{1,C|\xi|^{-(d-1)/2}\}
\eee
with the simple observation that $|1-\widehat{\psi}(\xi)|\leq\min\{1,C|\xi|\}$ gives
\bee
|\widehat{\sigma}(c\lm\xi)|=|\widehat{\sigma}(c\lm\xi)|\widehat{\psi}(\VE^4\lm\xi)+|\widehat{\sigma}(c\lm\xi)|(1-\widehat{\psi}(\VE^4\lm\xi))\leq \widehat{\psi}(\VE^4\lm\xi)+O(\min\{\VE^4\lm|\xi|,(c\lm|\xi|)^{-1/2}\}).
%=\widehat{\psi}(\VE^4\lm\xi)+O(\VE^{4/3}).
\eee

The result now follows, since $\|f_j\|_2^2\leq1$,
\bee
\min\{\VE^4\lm|\xi|,(c\lm|\xi|)^{-1/2}\}\leq c^{-1/3}\VE^{4/3}
\eee
 and a further application of Parseval (and appeal to (\ref{almostU1})) reveals that
\bee
\int |\widehat{f_j}(\xi)|^2\widehat{\psi}(\VE^4\lm\xi)\,d\xi=\iint f_j(x)f_j(x-x_1)\psi_{\VE^4\lm}(x_1)\,dx_1\,dx=\|f_j\|_{U^1(\VE^4\lm)}^2+O(\VE^4\lm).\qedhere
\eee
\end{proof}

%\newpage

\section{A New Proof of Theorem \ref{BourSimp}}\label{newsimplices}

In light of the reduction argument presented in Section \ref{newreduction} it is clear that in order to prove Theorem \ref{BourSimp} it would suffice to establish the following result for uniformly distributed subsets of $[0,1]^d$.

\begin{propn}[Simplices in uniformly distributed sets]\label{Propn00}
Let $\Delta_k=\{0,v_1,\dots,v_k\}$ be a fixed non-degenerate $k$-dimensional simplex with $c_{\Delta_k}=\min_{1\leq j\leq k}\text{\emph{dist}}(v_j,\text{\emph{span}}\left\{\{v_1,\dots,v_{k}\}\setminus v_j\right\})\leq 1.$

Let $0<\lm\leq\VE\ll \min\{1,c_{\Delta_k}^{-1}\}$
 and $A\subseteq[0,1]^d$ with $d\geq k+1$ and  $\A=|A|>0$.
If $A$ is $(\VE,\VE^4\lm)$-uniformly distributed, then
$A$ contains an isometric copy of $\lm\cdot\Delta_k$ and in fact
\be\label{intcount}
\iint 1_A(x)1_A(x-\lm \cdot U(v_1))\cdots 1_A(x-\lm\cdot U(v_k))\,d\mu(U) \,dx=\A^{k+1}+O_k(c_{\Delta_k}^{-1/6}\VE^{2/3})
\ee
where $\mu$ denotes the Haar measure on $SO(d)$.\end{propn}

Note that Proposition \ref{Propn0} is the special case of Proposition \ref{Propn00} with $k=1$ and $v_1=1$.

\subsection{Proof of Proposition \ref{Propn00}}\label{new2}
Let $\Delta_k=\{0,v_1,\dots,v_k\}$ be a fixed non-degenerate $k$-dimensional simplex with \[c_{\Delta_k}=\min_{1\leq j\leq k}\text{dist}(v_j,\text{span}\left\{\{v_1,\dots,v_{k}\}\setminus v_j\right\})\leq 1.\]

\begin{defn}[Counting Function for Simplices]

For any $0<\lm\ll1$ and functions
\[f_0,f_1,\dots,f_{k}:[0,1]^{d}\to\R\] with $d\geq k+1$ we  define
\be
T_{\Delta_k}(f_0,f_1,\dots,f_{k})(\lm)= \iint f_0(x)f_1(x-\lm \cdot U(v_1))\cdots f_k(x-\lm\cdot U(v_k))\,d\mu(U) \,dx.
%\idotsint\limits_{x_1,\dots,x_k\in[0,1]^d} f_1(x_1)\cdots f_k(x_k)   \,\Omega^{(k-1)}_{\lm}(x_1,\dots,x_k)\,dx_1\cdots\,dx_k
\ee

%\be
%\Omega^{(k-1)}_{\lm}(x_1,\dots,x_k)=\prod_{j=2}^{k}\sigma_{\lm}^{x_1,\dots,x_{j-1}}(x_j-x_1).
%\ee
\end{defn}
%We will assume that $|v_1|=1$.

%\be
%\Omega^{(k-1)}_{\lm}(x_1,\dots,x_k)=\prod_{j=2}^{k}\sigma_{\lm}^{x_1,\dots,x_{j-1}}(x_j-x_1).
%\ee

Proposition \ref{Propn00} is an immediate consequence of  the following ``generalized von-Neumann inequality''.

\begin{lem}[Generalized von-Neumann for Simplices]\label{GvN00}
For any $0<\VE, \lm\ll1$ and functions
\[f_0,f_1,\dots,f_{k}:[0,1]^{d}\to[-1,1]\]
\bee
\left|T_{\Delta_k}(f_0,f_1,\dots,f_{k})(\lm)\right|\leq\min_{j=0,1,\dots,k}\|f_j\|_{U^1(\VE^4\lm)}+O(c_{\Delta_k}^{-1/6}\VE^{2/3}).
\eee
\end{lem}

Indeed, if  $A\subseteq[0,1]^d$ with $d\geq k+1$ and $\A=|A|>0$, then Lemma \ref{GvN00} implies
\[\left|T_{\Delta_k}(1_A,\dots,1_A)(\lm)-T_{\Delta_k}(\A1_{[0,1]^d},\dots,\A1_{[0,1]^d})(\lm)\right|\leq (2^{k+1}-1)\|f_A\|_{U^1(\VE^4\lm)}+O_k(c_{\Delta_k}^{-1/6}\VE^{2/3})\]
for any $0<\VE,\lm\ll1$. Since
$T_{\Delta_k}(\A1_{[0,1]^d},\dots,\A1_{[0,1]^d})(\lm)=\A^{k+1}+O(\lm)$ it follows that
\bee
T_{\Delta_k}(1_A,\dots,1_A)(\lm)=\A^{k+1}+O_k(c_{\Delta_k}^{-1/6}\VE^{2/3})\eee
provided $0<\lm\leq\VE\ll 1$.

%\[\left|T_{\Delta_k}(1_A,\dots,1_A)(\lm)-T_{\Delta_k}(\A1_{[0,1]^d},\dots,\A1_{[0,1]^d})(\lm)\right|\leq (2^{k+1}-1)\|f_A\|^{1/2}_{U^1(\VE^4\lm)}+O_k(\VE^{1/4})\]
%for any $0<\VE,\lm\ll1$. Since
%$T(\A1_{[0,1]^d},\dots,\A1_{[0,1]^d})(\lm)=\A^{k+1}+O(\lm)$ it follows that
%\be T(1_A,\dots,1_A)(\lm)=\A^{k+1}+O_k(\VE^{1/4})\ee
%provided $0<\lm\leq\VE\ll1$.

\smallskip

To finish the proof of Proposition \ref{Propn00} we are therefore left with the task of proving Lemma \ref{GvN00}.

\begin{proof}[Proof of Lemma \ref{GvN00}]
By symmetry it suffices to show that
\be\label{now1}
\left|T_{\Delta_k}(f_0,f_1,\dots,f_{k})(\lm)\right|\leq \|f_k\|_{U^1(\VE^4\lm)}+O(c_{\Delta_k}^{-1/6}\VE^{2/3}).
\ee

As in \cite{B} we start by writing
\[
T_{\Delta_k}(f_0,f_1,\dots,f_{k})(\lm)=\iint\cdots\int f_0(x)f_{1}(x-\lm x_1)\cdots f_k(x-\lm x_k)\,d\sigma^{(d-k)}_{x_1,\dots,x_{k-1}}(x_k)\cdots d\sigma^{(d-2)}_{x_1}(x_2)\,d\sigma(x_{1})\,dx
\]
where $\sigma$ now denotes the normalized measure on the sphere $S^{d-1}(0,|v_1|)$ and
$\sigma^{(d-j)}_{x_1,\dots,x_{j-1}}$
denotes, for each $2\leq j\leq k$, the normalized measure on the spheres \be
S^{d-j}_{x_1,\dots,x_{j-1}}=S^{d-1}(0,|v_j|)\cap S^{d-1}(x_1,|v_j-v_1|)\cap\cdots\cap S^{d-1}(x_{j-1},|v_j-v_{j-1}|)\ee
where $S^{d-1}(x,r)=\{x'\in\R^d\,:\,|x-x'|=r\}$.
Since
\[
\left|T_{\Delta_k}(f_0,f_1,\dots,f_{k})(\lm)\right|\leq\iint\cdots\int\Bigl|  \int f_k(x-\lm x_k)\,d\sigma^{(d-k)}_{x_1,\dots,x_{k-1}}(x_k) \Bigr|\,
d\sigma^{(d-k+1)}_{x_1,\dots,x_{k-2}}(x_{k-1})\cdots d\sigma^{(d-2)}_{x_1}(x_2)\,d\sigma(x_{1})\,dx
\]
it follows from an application of Cauchy-Schwarz that
\begin{align}\label{CS}
\left|T_{\Delta_k}(f_0,f_1,\dots,f_{k})(\lm)\right|^2\leq\int\cdots\iint\Bigl|  \int f_k(x-\lm x_k) & \,d\sigma^{(d-k)}_{x_1,\dots,x_{k-1}}(x_k) \Bigr|^2\,dx
\\
%\,\,
& \,d\sigma^{(d-k+1)}_{x_1,\dots,x_{k-2}}(x_{k-1})\cdots d\sigma^{(d-2)}_{x_1}(x_2)\,d\sigma(x_{1}).\nonumber
\end{align}

An application of Plancherel therefore shows that
\[ \left|T_{\Delta_k}(f_0,f_1,\dots,f_{k})(\lm)\right|^2 \leq
\int|\widehat{f_k}(\xi)|^2I(\lm\,\xi)\,d\xi\]
where
\be\label{I}
I(\xi)=\int\cdots\int
\bigl|\widehat{\sigma^{(d-k)}_{x_1,\dots,x_{j-1}}}(\xi)\bigr|^2\,d\sigma^{(d-k+1)}_{x_1,\dots,x_{k-2}}(x_{k-1})\cdots d\sigma^{(d-2)}_{x_1}(x_2)\,d\sigma(x_{1}).
\ee

Estimate (\ref{now1}) will follow if we can show that
\be\label{key}
I(\lm\xi)=I(\lm\xi)\widehat{\psi}(\VE^4\lm\xi)+I(\lm\xi)(1-\widehat{\psi}(\VE^4\lm\xi))\leq \widehat{\psi}(\VE^4\lm\xi)+O(c_{\Delta_k}^{-1/3}\VE^{4/3})
\ee
since $\|f_k\|_2\leq1$ and an application of Parseval and appeal to (\ref{almostU1}) reveals that
\be\label{72}
\int |\widehat{f_k}(\xi)|^2\widehat{\psi}(\VE^4\lm\xi)\,d\xi=\iint f_k(x)f_k(x-x_1)\psi_{\VE^4\lm}(x_1)\,dx\,dx_1=\|f_k\|_{U^1(\VE^4\lm)}^2+O(\VE^4\lm).
\ee

To establish (\ref{key}) we argue as in \cite{B}, in particular we use the fact that in addition to being trivially bounded by 1 the Fourier transform of $\sigma^{(d-k)}_{x_1,\dots,x_{k-1}}$ also decays for large $\xi$ in certain directions, specifically
\be\label{decay}
\bigl|\widehat{\sigma^{(d-k)}_{x_1,\dots,x_{k-1}}}(\xi)\bigr|\leq C \left(r(S^{d-k}_{x_1,\dots,x_{k-1}})\cdot\text{dist}(\xi,\text{span}\{x_1,\dots,x_{k-1}\})\right)^{-(d-k)/2}
\ee
where $r(S^{d-k}_{x_1,\dots,x_{k-1}})=\text{dist}(v_k,\text{span}\{v_1,\dots,v_{k-1}\})$ denotes the radius of the sphere $S^{d-k}_{x_1,\dots,x_{k-1}}$.

This estimate is a consequence of the well-known asymptotic behavior of the Fourier transform of the measure on the unit sphere $S^{d-k}\subseteq\R^{d-k+1}$ induced by Lebesgue measure, see for example \cite{Stein}.

Together with the trivial uniform bound $I(\xi)\leq 1$, and an appropriate conical decomposition (depending on $\xi$) of the configuration space over which the integral $I(\xi)$ is defined, this gives
\be\label{Ibound}
I(\xi)\leq  \min\{1,C(c_{\Delta_k}|\xi|)^{-(d-k)/2}\}.%\ll\eta^{(d-j)/2}\ll\eta^{1/2}
\ee

Combining (\ref{Ibound}) with the basic bound
$
|1-\widehat{\psi}(\xi)|\leq\min\{1,C|\xi|\}
$
we obtain the uniform bound
\bee
|1-\widehat{\psi}(\VE^4\lm\,\xi)|I(\lm\,\xi)\ll\min\{(\lm c_{\Delta_k}|\xi|)^{-1/2},\VE^4\lm|\xi|\}\leq c_{\Delta_k}^{-1/3}\VE^{4/3}
\eee
from which (\ref{key}) follows.\qedhere
\end{proof}

\subsection{A Second New Proof of Theorem \ref{BourSimp}}

In this subsection we present an alternative approach to proving Proposition \ref{Propn00}  with the slightly worse error bound $O_k(c_{\Delta_k}^{-1/12}\VE^{1/3})$. Specifically, we show that one can in fact establish the following (slightly weaker)  generalized von-Neumann inequality for simplices using only Lemma \ref{GvN0}, namely the generalized von-Neumann inequality for distances.

\begin{lem}[Generalized von-Neumann for Simplices II]\label{GvN000}
For any $0<\lm\leq\VE\ll 1$ and functions
\[f_0,f_1,\dots,f_{k}:[0,1]^{d}\to[-1,1]\]
\bee
\left|T_{\Delta_k}(f_0,f_1,\dots,f_{k})(\lm)\right|\leq \sqrt{2\pi} \min_{j=0,1,\dots,k}\|f_j\|^{1/2}_{U^1(\VE^4\lm)}+O(c_{\Delta_k}^{-1/12}\VE^{1/3}).
\eee
\end{lem}

In the proof below we will make use of the following straightforward observations:
\begin{itemize}
\item[(i)] If we let $\Delta_{k-1}=\{0,v_1,\dots,v_{k-1}\}$, then
\be\label{i}
T_{\Delta_k}(f_0,f_1,\dots,f_{k-1},1_{[0,1]^d})(\lm)= T_{\Delta_{k-1}}(f_0,f_1,\dots,f_{k-1})(\lm) + O(\lm).
\ee
\item[(ii)] If we let $\Delta'_{k}=\{0,v'_1,\dots,v'_{k}\}$ with $v'_j=v_{k-j}-v_k$ for $0\leq j\leq k-1$ and $v'_k=-v_k$, then
\be\label{ii}
T_{\Delta_k}(f_0,f_1,\dots,f_{k})(\lm)= T_{\Delta'_{k}}(f_k,f_{k-1},\dots,f_{0})(\lm).
\ee
\end{itemize}

\begin{proof}[Proof of Lemma \ref{GvN000}]
By symmetry it suffices to show that
\be\label{now}
\left|T_{\Delta_k}(f_0,f_1,\dots,f_{k})(\lm)\right|^2\leq 2\pi\, \|f_k\|_{U^1(\VE^4\lm)}+O(c_{\Delta_k}^{-1/6}\VE^{2/3}).
\ee

We initially follow the proof of Lemma \ref{GvN00}, but after (\ref{CS}) %our application of Cauchy-Schwarz
we now proceed differently. Instead of applying Plancherel to the right hand side of
\bee
|T_{\Delta_k}(f_0,f_1,\dots,f_{k})(\lm)|^2\leq\iint\cdots\int\Bigl|  \int f_k(x-\lm x_k)\,d\sigma^{(d-k)}_{x_1,\dots,x_{k-1}}(x_k) \Bigr|^2
\,
d\sigma^{(d-k+1)}_{x_1,\dots,x_{k-2}}(x_{k-1})\cdots d\sigma(x_{1})\,dx.
\eee
we  now ``square out" the right hand side to obtain
\be\label{82}
\iint\!\!\cdots\!\!\iiint  \!\! f_k(x-\lm x_{k}) f_k(x-\lm x_{k+1})\,d\sigma^{(d-k)}_{x_1,\dots,x_{k-1}}(x_{k+1})\,d\sigma^{(d-k)}_{x_1,\dots,x_{k-1}}(x_{k})\,
d\sigma^{(d-k+1)}_{x_1,\dots,x_{k-2}}(x_{k-1})\cdots d\sigma(x_{1})\,dx.
\ee
%and show that the result can be reduced one for distances, namely Lemma \ref{GvN0}.

%Although not strictly necessary, we now proceed differently depending on whether $d=k+1$ or $d\geq k+2$.

If $d= k+1$,
 then for fixed $x_1,\dots, x_k$ we can use arc-length to parameterize of the circle $S^{d-k}_{x_1,\dots,x_{k-1}}$, with $\theta=0$ and $\theta=2\pi$ corresponding to the point $x_k$, to write
\be
\int f_k(x-\lm x_{k+1})\,d\sigma^{(d-k)}_{x_1,\dots,x_{k-1}}(x_{k+1})=\int_0^{2\pi}  f_k(x-\lm x_{k+1}(x_1,\dots,x_k,\theta))\,d\theta.
\ee

%Notice that the variable $\theta$ is not playing exactly  the same role as it did in the case when $d\geq k+2$ above, but this should not cause any confusion.
For any fixed $\theta\in[0,2\pi]$ we then define $\Delta_{k+1}(\theta)=\{0,v_1,\dots,v_k,v_{k+1}(\theta)\}$
with  $v_{k+1}=v_{k+1}(\theta)$ satisfying
$
|v_{k+1}|=|v_k|,
$
$
|v_{k+1}-v_j|=|v_k-v_j|
$
for all $1\leq j\leq k-1$ and use $\theta$ to determine the angle between $v_{k+1}$ and $v_k$ measured from the center of the circle $S^{d-k}_{x_1,\dots,x_{k-1}}$, consequently
\bee |v_{k+1}-v_k|=2\sin(\theta/2)\cdot \text{dist}(v_k,\text{span}\{v_1,\dots,v_{k-1}\}).
\eee

It follows that
\bee
|T_{\Delta_k}(f_0,f_1,\dots,f_{k})(\lm)|^2\leq\int_0^{2\pi} T_{\Delta_{k+1}(\theta)}(1_{[0,1]^d},\dots,1_{[0,1]^d},f_{k},f_{k})(\lm)\,d\theta + O(\lm)
\eee
and in light of (\ref{i}) and (\ref{ii}) that
\begin{align*}
|T_{\Delta_k}(f_0,f_1,\dots,f_{k})(\lm)|^2
&\leq\int_0^{2\pi}T_{\Delta'_{k+1}(\theta)}(f_{k},f_{k},1_{[0,1]^d},\dots,1_{[0,1]^d})(\lm)\,d\theta + O(\lm)\\
&=\int_0^{2\pi}T_{\Delta'_{1}(\theta)}(f_{k},f_{k})(\lm)\,d\theta + O(\lm)
\end{align*}
%\vspace{-5pt}
where
\[
T_{\Delta'_{1}(\theta)}(f_{k},f_{k})(\lm)=T(f_{k},f_{k})(c(\theta)\lm):=\iint f_k(x)f_k(x- c(\theta)\lm x_1)\,d\sigma(x_1) \,dx
\]
with $c(\theta)=2\sin(\theta/2)\cdot \text{dist}(v_k,\text{span}\{v_1,\dots,v_{k-1}\})$. Lemma \ref{GvN0} now implies that
\bee
|T_{\Delta'_{1}(\theta)}(f_{k},f_{k})(\lm)|\leq\|f_k\|_{U^1(\VE^4\lm)}+O((\sin(\theta/2))^{-1/6}c_{\Delta_k}^{-1/6}\VE^{2/3})
\eee
since $c(\theta)\geq 2\sin(\theta/2)\,c_{\Delta_k}$. This completes the proof, when $d=k+1$, as
$\int_0^{2\pi}(\sin(\theta/2))^{-1/6}\,d\theta<\infty,
$
and in fact establishes the result in general, since if $d\geq k+2$, one can define a new non-degenerate simplex \[\Delta_{d-1}=\{0,v_1,\dots,v_{k-1}, v'_{k},\dots,v'_{d-2},v'_{d-1}\}\]
with $v'_{d-1}=v_k$ and use the fact that
\bee
T_{\Delta_k}(f_0,f_1,\dots,f_{k})(\lm)=T_{\Delta_{d-1}}(f_0,\dots,f_{k-1},1_{[0,1]^d},\dots,1_{[0,1]^d}, f_{k})(\lm)+O(\lm).\qedhere
\eee
\end{proof}

\subsection{A Direct proof of Lemma \ref{GvN000} when $d\geq k+2$}\label{direct}
We choose to include an additional argument similar to the one presented above that covers the case $d\geq k+2$ directly. Arguments of this nature will be critical important  in Section \ref{6.2} when we establish a ``relative generalized von-Neumann inequality" for simplices.

\smallskip

If $d\geq k+2$ then in (\ref{82}), for fixed $x_1,\dots, x_k$, we write
\be
\sigma^{(d-k)}_{x_1,\dots,x_{k-1}}(x_{k+1})=\int_0^{\pi} (\sin\theta)^{d-k-1}\,d\sigma^{(d-k-1)}_{x_1,\dots,x_{k-1},x_k,\theta}(x_{k+1})\,d\theta
\ee
where $\sigma^{(d-k-1)}_{x_1,\dots,x_{k-1},x_k,\theta}(x_{k+1})$ denotes the normalized measure on the sphere \be
S^{d-k-1}_{x_1,\dots,x_{k-1},x_k,\theta}=S^{d-1}(0,|v_{k+1}|)\cap S^{d-1}(x_1,|v_{k+1}-v_1|)\cap\cdots\cap S^{d-1}(x_{k},|v_{k+1}-v_{k}|)\ee
with $v_{k+1}=v_{k+1}(\theta)$ defined such that
$
|v_{k+1}|=|v_k|,
$
$
|v_{k+1}-v_j|=|v_k-v_j|
$
for all $1\leq j\leq k-1$ with $\theta$ determining the angle between $v_{k+1}$ and $v_k$ measured from the center of the sphere $S^{d-k}_{x_1,\dots,x_{k-1}}$, consequently
\bee |v_{k+1}-v_k|=2\sin(\theta/2)\cdot \text{dist}(v_k,\text{span}\{v_1,\dots,v_{k-1}\}).
\eee
If we again let $\Delta_{k+1}(\theta)=\{0,v_1,\dots,v_k,v_{k+1}\}$, it follows that
\bee
|T_{\Delta_k}(f_0,f_1,\dots,f_{k})(\lm)|^2\leq\int_0^\pi (\sin\theta)^{d-k-1}T_{\Delta_{k+1}(\theta)}(1_{[0,1]^d},\dots,1_{[0,1]^d},f_{k},f_{k})(\lm)\,d\theta + O(\lm)
\eee
and in light of (\ref{i}) and (\ref{ii}) that
\begin{align*}
|T_{\Delta_k}(f_0,f_1,\dots,f_{k})(\lm)|^2
&\leq\int_0^\pi (\sin\theta)^{d-k-1}T_{\Delta'_{k+1}(\theta)}(f_{k},f_{k},1_{[0,1]^d},\dots,1_{[0,1]^d})(\lm)\,d\theta + O(\lm)\\
&=\int_0^\pi (\sin\theta)^{d-k-1}T_{\Delta'_{1}(\theta)}(f_{k},f_{k})(\lm)\,d\theta + O(\lm)
\end{align*}
where again
\bee
T_{\Delta_{1}(\theta)}(f_{k},f_{k})(\lm)=T(f_{k},f_{k})(c(\theta)\lm):=\iint f_k(x)f_k(x- c(\theta)\lm x_1)\,d\sigma(x_1) \,dx
\eee
with $c(\theta)=2\sin(\theta/2)\cdot \text{dist}(v_k,\text{span}\{v_1,\dots,v_{k-1}\})$. Lemma \ref{GvN0} again implies that
\bee
|T_{\Delta'_{1}(\theta)}(f_{k},f_{k})(\lm)|\leq\|f_k\|_{U^1(\VE^4\lm)}+O((\sin(\theta/2))^{-1/6}c_{\Delta_k}^{-1/6}\VE^{2/3})
\eee
since $c(\theta)\geq 2\sin(\theta/2)\,c_{\Delta_k}$ and this completes the proof as
$\int_0^\pi (\sin\theta)^{d-k-1}(\sin(\theta/2))^{-1/6}\,d\theta<\infty.$
\qed

%\end{proof}

\bigskip

%\newpage

\section{Proof of Theorems \ref{Rect} and \ref{ProdSimp}}\label{4}

We now proceed with the main task, namely the proofs of Theorems  \ref{Rect} and \ref{ProdSimp}.

\subsection{Reducing Theorems \ref{Rect} and \ref{ProdSimp} to quantitative results
%[for the product of simplices in]
for subsets of $[0,1]^{d_1}\times[0,1]^{d_2}$}\label{red}

\begin{propn}[Rectangles]\label{Propn1}
Let $0<c\leq1$ and $A\subseteq[0,1]^{d_1}\times[0,1]^{d_2}$ with $d_1,d_2\geq2$ and  $\A=|A|>0$.

If $\{\lm_j\}$ is any sequence in $(0,1)$ with $\lm_{j+1}<\frac{1}{2}\lm_j$ for all $j\geq1$, then there exist $1\leq j\leq J(\A)$ and a quadruple of points
\bee\{(x,y), (x',y), (x,y'), (x',y')\}\subseteq A\quad\text{with}\quad|x-x'|=\lm_j\text{ and\, }|y-y'|=c\lm_j.\eee

In fact, for $\lm=\lm_j$
\bee
\iiiint 1_A(x,y)1_A(x-\lm x_1,y)1_A(x,y-c\lm y_1)1_A(x-\lm x_1,y-c\lm y_1)\,d\sigma_1(x_1)\,d\sigma_2(y_1)\,dx\,dy\geq C(\A)>0
\eee
%\prod_{\omega_1,\omega_2\in\{0,1\}}
where $\sigma_i$
denotes, for $i=1,2$, the normalized measure on the unit sphere $S^{d_i-1}\subseteq\R^{d_i}$ centered at the origin induced by the Lebesgue measure on $\R^{d_i}$.
\end{propn}

\begin{propn}[Product of Simplices]\label{Propn11}
Let $\Delta_{k_i}=\{0,v^i_1,v^i_2,\dots,v^i_{k_i}\}$ be fixed non-degenerate simplices of dimension $k_i$  with \[c_{\Delta_{k_i}}=\min_{1\leq j\leq k_i}\text{\emph{dist}}(v^i_j,\text{\emph{span}}\left\{\{v^i_1,\dots,v^i_{k_i}\}\setminus v^i_j\right\})\leq1\]
for $i=1,2$ and $A\subseteq[0,1]^{d_1}\times[0,1]^{d_2}$ with $d_i\geq k_i+3$ and $\A=|A|>0$.

If
 $\{\lm_j\}$ is any sequence in $(0,1)$ with  $\lm_{j+1}<\frac{1}{2}\lm_j$ for all $j\geq1$, then there exist $1\leq j\leq J(\A,\Delta_{k_1},\Delta_{k_2})$ and a product $\Delta_{k_1}'\times\Delta_{k_2}'\subseteq A$ with each $\Delta_{k_i}'\subseteq[0,1]^{d_i}$ an  isometric copy of $\lm_j\cdot\Delta_{k_i}$.
In fact, for $\lm=\lm_j$
\bee
\iiiint \prod_{i=0}^{k_1}\prod_{j=0}^{k_2}1_A(x-\lm\cdot U_1(v^1_i),y-\lm\cdot U_2(v^2_j))
\,d\mu_1(U_1)\,d\mu_2(U_2)\,dx\,dy\geq C(\A)>0
\eee
where $v_0^1=v_0^2=0$ and  $\mu_1$ and $\mu_2$ denote the Haar measures on $SO(d_1)$ and $SO(d_2)$ respectively.
\end{propn}

The reduction of Theorems \ref{Rect} and \ref{ProdSimp} to these results in the compact setting of $[0,1]^{d_1}\times[0,1]^{d_2}$ is straightforward and precisely the approach taken by Bourgain  in \cite{B} to prove Theorem \ref{BourSimp}, but for completeness we supply the details for Theorem \ref{Rect} below.

\begin{proof}[Proof that Proposition \ref{Propn1} implies Theorem \ref{Rect}]
%\subsection{Proof that Proposition \ref{Propn1} implies Theorem \ref{Thm1}}
We may assume that $c:=|v_2|\leq|v_1|=1$.

Arguing indirectly we suppose that $A\subseteq \R^d$ with $d\geq 4$ is a set with $\delta^*(A)>0$ for which the conclusion of Theorem \ref{Rect} fails to hold, namely that there exist arbitrarily large $\lm\in\R$ for which $A$ does not contain an isometric copy of $\lm \cdot\Box$.

We now let $0<\alpha<\delta^*(A)$ and set $J=J(\alpha)$ from Proposition \ref{Propn1}.  By our indirect assumption we can choose a sequence $\{\lm_j\}_{j=1}^J$ with the property that
$\lm_{j+1}<\frac{1}{2}\lm_j$ for all $1\leq j\leq J-1$ and $A$ does not contain an isometric copy of $\lm_j \cdot\Box$ for each $1\leq j\leq J$.
It follows from the definition of upper Banach density that exist $N\in\R$ with $N\gg\lm_1$ and $t_0\in\R^d$ for which
\[\frac{|A\cap(t_0+Q_N)|}{|Q_N|}\geq\alpha.\]
Rescaling $A\cap(t_0+Q_N)$ to a subset of $[0,1]^d$ and applying Proposition \ref{Propn1} leads to a contradiction. %\qed
\end{proof}

\subsection{Proof of Propositions \ref{Propn1} and \ref{Propn11}, Part I: A Density Increment Strategy %: \\ A Generalized von-Neumann Inequality and Inverse Theorem for the *?*$\Box_{B_1,B_2}(L)$-norm
}\label{SecPart1}

%Throughout this section we fix $c\geq1$ and let $B_i\subseteq[0,1]^{d_i}$ with $d_i\geq 2$ denote arbitrary sets with $\B_i=|B_i|>0$, for $i=1,2$.
%Furthermore, on each of the sets $B_i$  we introduce the normalized measures
%\be\mu_{B_i}=\B_i^{-1}1_{B_i}.\ee
%We will  establish the following partial progress towards proving Proposition \ref{Propn1}. %which we will ultimately achieve via a density increment strategy based on the following

\begin{propn}[Dichotomy for Rectangles]\label{part1}

Let $0<c\leq1$ and $B_i\subseteq[0,1]^{d_i}$ with $d_i\geq 2$  and $\B_i=|B_i|>0$ for $i=1,2$.
%Let $0<\A,\B_1,\B_2\leq 1$  and $0<\lm\leq\VE\ll\B_1^{6}\B_2^{6}\A^{32}$.
If $A\subseteq B_1\times B_2$ with $|A|=\A\B_1\B_2>0$ and $0<\lm\leq\VE\ll c\B_1^{6}\B_2^{6}\A^{32}$, then either
\bee
\frac{1}{\B_1^{2}\B_2^{2}}\iiiint 1_A(x,y)1_A(x-\lm x_1,y)1_A(x,y-c\lm y_1)1_A(x-\lm x_1,y-c\lm y_1)\,d\sigma_1(x_1)\,d\sigma_2(y_1)\,dx\,dy\geq\frac{1}{2}\A^{4}\eee
or there exist cubes $Q_i\subseteq[0,1]^{d_i}$ of side-length $\VE^4\lm$, sets $B_i'$  in $Q_i$, and $c'>0$ for which
\bee
\frac{|A\cap(B_1'\times B_2')|}{|B_1'\times B_2'|}\geq \A+c'\,\A^{32}.
\eee
provided $B_1$ and $B_2$ are $(\VE,\VE^4\lm)$-uniformly distributed subsets of $[0,1]^{d_1}$ and $[0,1]^{d_2}$ respectively.
\end{propn}

%Proposition \ref{part1} will be an immediate consequence of Corollaries \ref{CorGvN1} and \ref{InvCor} below, which are in turn consequences of an appropriate ``Generalized von-Neumann Inequality'' and  companion ``Inverse Theorem".

\begin{propn}[Dichotomy for Product of Simplices]\label{part11}

For $i=1,2$
let $B_i\subseteq[0,1]^{d_i}$ with $d_i\geq k_i+3$ and $\B_i=|B_i|>0$  and
$\Delta_{k_i}=\{v_0^i,v^i_1,v^i_2,\dots,v^i_{k_i}\}$ be a non-degenerate simplex of dimension $k_i$ with $v_0^i=0$
and \[c_{\Delta_{k_i}}=\min_{1\leq j\leq k_i}\text{\emph{dist}}(v^i_j,\text{\emph{span}}\left\{\{v^i_1,\dots,v^i_{k_i}\}\setminus v^i_j\right\})\leq1.\]

If $A\subseteq B_1\times B_2$ with $|A|=\A\B_1\B_2>0$ and %$0<\lm\leq\VE\ll(\B_1^{k_1+1}\B_2^{k_2+1})^6\A^{8(k_1+1)(k_2+1)}$, 
\[0<\lm\leq\VE\ll_{k_1,k_2}(c_{\Delta_{k_1}}c_{\Delta_{k_2}})^{2}(\B_1^{k_1+1}\B_2^{k_2+1}\A^{(k_1+1)(k_2+1)})^{16}\]
then either
\bee
\frac{1}{\B_1^{k_1+1}\B_2^{k_2+1}}\iiiint \prod_{i=0}^{k_1}\prod_{j=0}^{k_2}1_A(x-\lm\cdot U_1(v^1_i),y-\lm\cdot U_2(v^2_j))
\,d\mu_1(U_1)\,d\mu_2(U_2)\,dx\,dy\geq\frac{1}{2}\A^{(k_1+1)(k_2+1)}\eee
or there exist cubes $Q_i\subseteq[0,1]^{d_i}$ of side-length $\VE^4\lm$, sets $B_i'$  in $Q_i$, and $c'>0$ for which
\bee
\frac{|A\cap(B_1'\times B_2')|}{|B_1'\times B_2'|}\geq \A+c'\,\A^{8(k_1+1)(k_2+1)}.
\eee
provided $B_1$ and $B_2$ are $(\VE,\VE^4\lm)$-uniformly distributed subsets of $[0,1]^{d_1}$ and $[0,1]^{d_2}$ respectively.
\end{propn}

Sections \ref{SecPart1} and \ref{SecPart11} below are devoted to the proofs of
Propositions \ref{part1} and \ref{part11}. Central to each proof is an appropriate ``relative generalized von-Neumann inequality'', namely Lemmas \ref{GvN1} and \ref{GvN11}. These relative generalized von-Neumann inequalities in turn imply Corollaries \ref{CorGvN1} and \ref{CorGvN11}, which together with Corollaries \ref{InvCor1} and \ref{InvCor11} (which are both consequences of an appropriate common ``Inverse Theorem", namely Theorem \ref{InvThm}) immediately imply Propositions \ref{part1} and \ref{part11} respectively.

It is important to note that Propositions \ref{part1} and \ref{part11} are not in and of themselves sufficient to establish Propositions \ref{Propn1} and \ref{Propn11}. In order to apply a density increment argument one would need that the sets $B_1'$ and $B_2'$ produced by Propositions \ref{part1} and \ref{part11}, for which $A$ has increased density on $B_1'\times B_2'$, were $(\eta,L')$-uniformly distributed for a sufficiently small $\eta$ and for $L'$ attached to some of the $\lm_j$'s on $Q_1$ and $Q_2$ respectively, which they simply may not be.
In Section \ref{SecPart2} we complete the proofs of Proposition \ref{Propn1} and \ref{Propn11} by showing that we can obtain suitably uniformly distributed sets $B_1'$ and $B_2'$ by appealing to a version of Szemer\'edi's Regularity Lemma \cite{Sz}  adapted to a sequence of scales.

%\newpage

\section{Proof of Proposition \ref{part1}
%: \\ A Generalized von-Neumann Inequality and Inverse Theorem for the *?*$\Box_{B_1,B_2}(L)$-norm
}\label{SecPart1}

%Furthermore, on each of the sets $B_i$  we introduce the normalized measures
%\be\mu_{B_i}=\B_i^{-1}1_{B_i}.\ee
%We will   establish the following partial progress towards proving Proposition \ref{Propn1}. %which we will ultimately achieve via a density increment strategy based on the following

At the heart of our proof of Proposition \ref{part1} will be an appropriate ``relative generalized von-Neumann inequality for rectangles", namely Lemma \ref{GvN1} below.  This result, together with a companion ``Inverse Theorem" (Theorem \ref{InvThm} below) and Proposition \ref{Propn0} will ultimately furnish a proof of Proposition \ref{part1}.

Throughout this section we fix $B_i\subseteq[0,1]^{d_i}$ with $d_i\geq 2$ to be arbitrary sets with $\B_i=|B_i|>0$ for $i=1,2$.

%Our proof of Proposition \ref{Propn11} is significantly more involved than that of Proposition \ref{Propn00}.

\subsection{A Relative Generalized von-Neumann Inequality for Distances and Rectangles}

 \begin{defn}[A Counting Function for Rectangles]
For any $0<c\leq1$, $0<\lm\ll1$ and functions \[f_{ij}:[0,1]^{d_1}\times [0,1]^{d_2}\to\R\] with $i,j\in\{0,1\}$  we define
\[T_{\Box_c}(\lm):=T_{\Box_c}(f_{00},f_{10},f_{01},f_{11})(\lm)\]
where
\be
T_{\Box_c}(\lm)=
\iiiint f_{00}(x,y)f_{10}(x-\lm x_1,y)f_{01}(x,y-c\lm y_1)f_{11}(x-\lm x_1,y-c\lm y_1)\,d\sigma_1(x_1)\,d\sigma_2(y_1)\,dx\,dy
\ee
\end{defn}

Note that if we let
\be
\nu(x,y)=\nu_1(x)^{1/2}\nu_2(y)^{1/2}
\ee
where
\be
\nu_1=\B_1^{-1}1_{B_1}\quad\text{and}\quad \nu_2=\B_2^{-1}1_{B_2}
\ee
then,
in light of Proposition \ref{Propn0}, we have
\be\label{observation}
T_{\Box_c}(\nu,\nu,\nu,\nu)(\lm)=T(\nu_1,\nu_1)(\lm)\cdot T(\nu_2,\nu_2)(c\lm)=1+O(\B_1^{-2}\B_2^{-2}c^{-1/6}\VE^{2/3})
\ee
for any $0<\lm\leq\VE\ll1$, provided $B_1$ and $B_2$ are $(\VE,\VE^4\lm)$-uniformly distributed subsets of $[0,1]^{d_1}$ and $[0,1]^{d_2}$ respectively.

\comment{
\begin{align*}
T_{B_1,B_2}&(f_{00},\dots,f_{k_1k_2})(\lm)\\
& =\idotsint \prod_{i=0}^{k_1}\prod_{j=0}^{k_2}f_{ij}(x_i,y_j)\,\Omega^{(k_1-1)}_{1,\lm}(\mathbf{x})\,\Omega^{(k_2-1)}_{2,\lm}(\mathbf{y})\,d\mu_{B_1}(x_1)\cdots d\mu_{B_1}(x_{k_1})\,d\mu_{B_2}(y_1)\cdots  d\mu_{B_2}(y_{k_2}).
\end{align*}
}

\begin{defn}[$\Box(L)$-norm]
%Let  $B_1,B_2\subseteq[0,1]^2$ be measurable with $\beta_j:=|B_j|>0$ for $j=1,2$.

For  $0<L\ll1$ and functions $f:[0,1]^{d_1}\times [0,1]^{d_2}\to\R$ we define
\be\label{B1}
\|f\|_{\Box(L)}^4=\int_{[0,1]^{d_1}}\int_{[0,1]^{d_2}}\|f\|_{\Box(L)(t_1,t_2)}^4\,dt_2\,dt_1
\ee
with
\be\label{B2}
\|f\|_{\Box(L)(t_1,t_2)}^4=\frac{1}{L^{2(d_1+d_2)}}\iiiint\limits_{\substack{x,x'\in t_1+Q_{1,L}\\ y,y'\in t_2+Q_{2,L}}} f(x,y)f(x',y)f(x,y')f(x',y')\,\,dx'\,dx\,dy'\,dy
\ee
where $Q_{i,L}=[-L/2,L/2]^{d_i}$ for $i=1,2$.
\end{defn}

As before it is a straightforward but important observation that $\|f\|_{\Box(L)}^4$ equals
\be\label{almost2}
%\|f\|_{\Box_{B_1,B_2}(L)}^4=
\iiiint f(x,y)f(x-x_1,y)f(x,y-y_1)f(x-x_1,y-y_1)
\psi_{1,L}(x_1)\psi_{2,L}(y_1)\,dx_1\,dx\,dy_1\,dy+O(L)
\ee
where
$
\psi_{i,L}=L^{-2d_i}\,1_{Q_{i,L}}*1_{Q_{i,L}}.
$

%\subsection{A Generalized von-Neumann Inequality}

\smallskip

In this setting we have the following ``generalized von-Neumann inequality" relative to $B_1\times B_2$.
%for which it is essential that our count of product simplices is taken relative to suitably uniformly distributed sets $B_1$ and $B_2$.

\begin{lem}[Generalized von-Neumann for Rectangles relative to $B_1\times B_2$]\label{GvN1} 

Let $0<c\leq1$ and
$
\nu=\nu_1^{1/2}\otimes\nu_2^{1/2}
$
where
$
\nu_1=\B_1^{-1}1_{B_1}$ and $\nu_2=\B_2^{-1}1_{B_2}
$.
For any 
 $0<\VE,\lm\ll1$ and functions
\[f_{ij}:[0,1]^{d_1}\times [0,1]^{d_2}\to[-1,1]\]
with $i,j\in\{0,1\}$ we have
\bee
|T_{\Box_c}(f_{00}\nu,f_{10}\nu,f_{01}\nu,f_{11}\nu)(\lm)|\leq \prod_{i,j\in\{0,1\}}\!\! \|f_{ij}\nu\|_{\Box(\VE^4\lm)} + O(\B_1^{-1}\B_2^{-1}c^{-1/24}\VE^{1/6}).
\eee
%provided $B_1$ and $B_2$ are $(\VE,\VE^4\lm)$-uniformly distributed subsets of $[0,1]^{d_1}$ and $[0,1]^{d_2}$ respectively.
\end{lem}

It is easy to see that Lemma \ref{GvN1}, combined with Proposition \ref{Propn0}, gives the following
\begin{cor}\label{CorGvN1}
Let $0<c\leq1$, $0<\A,\B_1,\B_2\leq 1$ and $0<\lm\leq\VE\ll c\B_1^{6}\B_2^{6}\A^{24}$.

If $A\subseteq B_1\times B_2\subseteq[0,1]^{d_1}\times[0,1]^{d_2}$ with $|A|=\A\B_1\B_2$ and %and $\B_i=|B_i|$ for $i=1,2$.
$
\|f_A\nu\|_{\Box(\VE^4\lm)}\ll\A^{4},
$
then
%with $f_A=1_A-\A1_{B_1\times B_2}$
\bee
T_{\Box_c}(1_A\nu,1_A\nu,1_A\nu,1_A\nu)(\lm)\geq\frac{1}{2}\A^{4}
\eee
provided $B_1$ and $B_2$ are $(\VE,\VE^4\lm)$-uniformly distributed subsets of $[0,1]^{d_1}$ and $[0,1]^{d_2}$ respectively.
\end{cor}

\begin{proof}[Proof of Corollary \ref{CorGvN1}]
It follows immediately from Lemma \ref{GvN1} that
\bee
\Bigl|T_{\Box_c}(1_A\nu,1_A\nu,1_A\nu,1_A\nu)(\lm)-\A^4T_{\Box_c}(\nu,\nu,\nu,\nu)(\lm)\Bigr|
\leq 15\,\|f_A\nu\|_{\Box(\VE^4\lm)}+ O(\B_1^{-1}\B_2^{-1}c^{-1/24}\VE^{1/6})
\eee
for any $0<\VE,\lm\ll1$, where $
f_A=1_A-\A1_{B_1\times B_2}
$.
 The result follows since, as noted in (\ref{observation}), the fact that $B_1$ and $B_2$ are $(\VE,\VE^4\lm)$-uniformly distributed subsets of $[0,1]^{d_1}$ and $[0,1]^{d_2}$ allows us to use Proposition \ref{Propn0} and conclude that
 \[T_{\Box_c}(\nu,\nu,\nu,\nu)(\lm)=1+O(\B_1^{-2}\B_2^{-2}c^{-1/6}\VE^{2/3})
\]
for any $0<\lm\leq\VE\ll1$, as required.
\end{proof}

\subsection{Proof of Lemma \ref{GvN1}}

The proof of Lemma \ref{GvN1} follows from two clever applications of Cauchy-Schwarz combined with the following relative version of Lemma \ref{GvN0}.

%\subsubsection{A version of Lemma \ref{GvN00} relative to uniformly distributed sets}

%\begin{proof}[Proof of Lemma \ref{GvN1}]
\begin{lem}[Relative Version of Lemma \ref{GvN0}]\label{GvN0withB}
Let $B\subseteq[0,1]^d$ with $d\geq2$ and $\B=|B|$.

For any $0<c\leq1$, $0<\VE, \lm\ll1$ and functions
$f_0,f_1:[0,1]^d\to[-1,1]$ we have
\bee
\left|T(f_0\nu,f_1\nu)(c\lm)\right|\leq\prod_{j\in\{0,1\}} \left(\iint f_j\nu(x)f_j\nu(x-x_1)\psi_{\VE^4\lm}(x_1)\,dx_1\,dx\right)^{1/2}\!\!+\,O(\B^{-1}c^{-1/6}\VE^{2/3}).
\eee
where $\nu=\B^{-1}1_B$.
\end{lem}

\begin{proof}
Same as that for Lemma \ref{GvN0} above, but noting that $\|f_j\nu\|_2^2\leq\B^{-1}$ for $j=0,1$.% (since $|f_j|\leq1_{B_i}$).
\end{proof}

To prove Lemma \ref{GvN1} we first observe that
\bee
\left|T_{\Box_c}(f_{00}\nu,f_{10}\nu, f_{01}\nu,f_{11}\nu)(\lm)\right|
\leq\iint \left|T(g_0^{x,x_1}\nu_2,g_1^{x,x_1}\nu_2)(c\lm)\right|\,\nu_1(x)\nu_1(x-\lm x_1) \,d\sigma_1(x_1)\,dx
\eee
where
\begin{align*}
g_0^{x,x_1}(y)&=f_{00}(x,y)f_{10}(x-\lm x_1,y)\\
g_1^{x,x_1}(y)&=f_{01}(x,y)f_{11}(x-\lm x_1,y).
\end{align*}

Applying Lemma \ref{GvN0withB} to $T(g_0^{x,x_1}\nu_2,g_1^{x,x_1}\nu_2)(c\lm)$ followed by an application of Cauchy-Schwarz (and switching the order of integration) shows that $|T_{\Box_c}(f_{00}\nu,\dots,f_{11}\nu)(\lm)|^2$ is majorized by
\[
%\left|T_{\Box_c}(f_{00}\nu,\dots,f_{11}\nu)(\lm)\right|^2
%\leq
\prod_{j\in\{0,1\}}\iint \left|T(h_{0j}^{y,y_1}\nu_1,h_{1j}^{y,y_1}\nu_1)(\lm)\right|\nu_2(y)\nu_2(y-\lm y_1) \psi_{2,\VE^4\lm}(y_1)\,dy_1\,dy+O(\B_1^{-2}\B_2^{-2}c^{-1/6}\VE^{2/3})
\]
where
\begin{align*}
h_{0j}^{y,y_1}(x)&=f_{0j}(x,y)f_{0j}(x,y-\lm y_1)\\
h_{1j}^{y,y_1}(x)&=f_{1j}(x,y)f_{1j}(x,y-\lm y_1).
\end{align*}

Applying Lemma \ref{GvN0withB} once more, this time to $T(h_{0j}^{y,y_1}\nu_1,h_{1j}^{y,y_1}\nu_1)(\lm)$, followed by another application of Cauchy-Schwarz reveals that
$\left|T_{\Box_c}(f_{00}\nu,\dots,f_{11}\nu)(\lm)\right|^4$ is majorized by
\[
%\left|T_{B_1,B_2}(f_{11},f_{21}, f_{12},f_{22})(\lm)\right|^4\leq
\prod_{i,j\in\{0,1\}}\iiiint h_{ij}^{y,y_1}\nu_1(x)h_{ij}^{y,y_1}\nu_1(x-x_1)
\,\nu_2(y)\nu_2(y-\lm y_1) \psi_{1,\VE^4\lm}(x_1)\psi_{2,\VE^4\lm}(y_1)\,dx_1\,dx\,dy_1\,dy+O(\B_1^{-4}\B_2^{-4}c^{-1/6}\VE^{2/3})
\]

Since
\bee
h_{ij}^{y,y_1}\nu_1(x)h_{ij}^{y,y_1}\nu_1(x-x_1)\nu_2(y)\nu_2(y-\lm y_1)=f_{ij}\nu(x,y)f_{ij}\nu(x-x_1,y)f_{ij}\nu(x,y-y_1)f_{ij}\nu(x-x_1,y-y_1)\eee
the result follows in light of observation (\ref{almost2}).
\qed
%\end{proof}

%\newpage

\medskip

\subsection{Inverse Theorem for the $\Box(L)$-norm}

The final piece in the proof of Proposition \ref{part1} is the following

\begin{thm}[Inverse Theorem]\label{InvThm} Let $0<\eta, \B_1,\B_2\leq1$ and $B_1$ and $B_2$ be $(\VE,L)$-uniformly distributed subsets of $[0,1]^{d_1}$ and $[0,1]^{d_2}$  with $0<L\leq\VE\ll \eta^8\,\B_1^{2}\B_2^{2}$.
If $f:[0,1]^{d_1}\times [0,1]^{d_2}\to[-1,1]$ satisfies
\be\label{assumption}
\iint f(x,y)\nu_1(x)\nu_2(y)\,dx\,dy=0\quad\quad\text{and}\quad\quad
\|f\nu\|_{\Box(L)}\geq\eta
\ee
with
$
\nu=\nu_1^{1/2}\otimes\nu_2^{1/2}
$
and
$
\nu_1=\B_1^{-1}1_{B_1}$ and $\nu_2=\B_2^{-1}1_{B_2}
$,
then
there exist cubes $Q_i\subseteq[0,1]^{d_i}$ of side-length $L$ and sets $B_i'\subseteq B_i\cap Q_i$  such that
\be\label{3.25}
\frac{1}{L^{d_1+d_2}}\iint_{B_1'\times B_2'} f(x,y)\nu_1(x)\nu_2(y)\,dx\,dy\geq c\,\eta^8.
\ee
\end{thm}

As a consequence of Theorem \ref{InvThm} we immediately obtain the following corollary which together with Corollary \ref{CorGvN1} implies Proposition \ref{part1}.

\begin{cor}\label{InvCor1}
Let $0<\A,\B_1,\B_2\leq 1$ and $B_1$ and $B_2$ be $(\VE,\VE^4\lm)$-uniformly distributed subsets of $[0,1]^{d_1}$ and $[0,1]^{d_2}$  with $0<\lm\leq\VE\ll\B_1^{2}\B_2^{2}\A^{32}$.

If $A\subseteq B_1\times B_2\subseteq[0,1]^{d_1}\times[0,1]^{d_2}$ with $|A|=\A\B_1\B_2$ and
\[\|f_A\nu\|_{\Box(\VE^4\lm)}\gg\A^{4}\]
with $f_A=1_A-\A1_{B_1\times B_2}$,
then there exist cubes $Q_i\subseteq[0,1]^{d_i}$ of side-length $\VE^4\lm$ and sets $B_i'$  in $Q_i$ for which
\bee
\frac{|A\cap(B_1'\times B_2')|}{|B_1'\times B_2'|}\geq \A+c\,\A^{32}.
\eee
\end{cor}

\begin{proof}[Proof of Theorem \ref{InvThm}]

If \eqref{3.25} holds for some cubes $Q_i:=t_i+Q_L$ and sets $B_i':=B_i\cap Q_i$, then Theorem \ref{InvThm} follows, so we may assume for all $t_1\in [0,1]^{d_1}$ and $t_2\in [0,1]^{d_2}$  that
\be\label{3.27}
I(t_1,t_2):=\frac{1}{\B_1\B_2L^{d_1+d_2}}\int_{t_1+Q_L}\int_{t_2+Q_L} f(x,y)\,dx\,dy\leq c\,\eta^8
\ee
with say $c=2^{-16}$.
It is then easy to see that this assumption, together with our assumption on the sets $B_i$, namely that
\[\int ||B_i\cap (t+Q_L)|-\B_i L^{d_i}|^2\,dt \leq \VE^2 L^{2d_i},\]
imply, via an easy averaging argument, that
\be\label{3.30}
|G_{\eta,\VE}|\geq \frac{\eta^4}{16}\quad\text{where}\quad
G_{\eta,\VE}=\left\{(t_1,t_2)\in G_\VE\,:\,  \|f\nu\|^4_{\Box(L)(t_1,t_2)}\geq \frac{\eta^4}{16}\right\}
\ee
and
\bee
G_\VE=\left\{(t_1,t_2);\ |B_i\cap (t_i+Q_L)-\B_i L^{d_i}|\leq \VE^{1/2} L^2\,\text{for $i=1,2$}\right\}.
\eee

We first show that if there exist $(t_1,t_2)\in G_{\eta,\VE}$ for which $|I(t_1,t_2)|\leq \eta^4/2^9$, then Theorem \ref{InvThm} holds. Indeed, by the pigeonhole principle, we see that given such a pair $(t_1,t_2)$ we may choose $x_1\in [0,1]^{d_1}$ and $y_1\in[0,1]^{d_2}$ so that
\be\label{3.31}
\left|\frac{1}{\B_1\B_2L^{d_1+d_2}} \int_{t_1+Q_L}\int_{t_2+Q_L} f(x_2,y_2)f(x_2,y_1)f(x_1,y_2)\,dx_2\,dy_2\,\right| \geq
\frac{\eta^4}{32}.
\ee
If we now write $f_{y_1}(x_2)=f(x_2,y_1)$, $f_{x_1}(y_2)=f(x_1,y_2)$ and decompose $f_{y_1}=f_{y_1}^+-f_{y_1}^-$ and $f_{x_1}=f_{x_1}^+-f_{x_1}^-$ into their respective positive and negative parts, then it follows that
\bee
\left|\frac{1}{\B_1\B_2L^{d_1+d_2}}\int_{t_1+Q_L}\int_{t_2+Q_L}f(x_2,y_2)g_1(x_2)g_2(y_2)\,dx_2\,dy_2\,\right| \geq
\frac{\eta^4}{2^7},
\eee
for some functions $g_i:[0,1]^{d_i}\to [0,1]$. Writing these functions as an average of indicator functions, namely
\[g_i(x)=\int_0^1 1_{\{g_i(x)\geq s\}}\,ds
%\quad\text{and}\quad h(y)=\int_0^1 1_{\{h(y)\geq s\}}\,ds
\]
and appealing again to the  pigeonhole principle, we see that we may choose sets $U_1$ and $V_1$ so that
\be\label{8}
\left|\frac{1}{\B_1\B_2L^{d_1+d_2}} \int_{t_1+Q_L}\int_{t_2+Q_L} f(x_2,y_2)1_{U_1}(x_2)1_{V_1}(y_2)\,dx_2\,dy_2\,\right| \geq
\frac{\eta^4}{2^7}.
\ee
We now set $U_2=U_1^c$, $V_2=V_1^c$ and define, for $j,j'\in \{1,2\}$, the integrals
\bee
I_{j,j'}:= \frac{1}{\B_1\B_2L^{d_1+d_2}} \int_{t_1+Q_L}\int_{t_2+Q_L}f(x_2,y_2)1_{U_j}(x_2)1_{V_{j'}}(y_2)\,dx_2\,dy_2.\eee
Note that we know $|I_{1,1}|\geq \eta^4/2^7$ and if $I_{1,1}\geq \eta^4/2^7$ then \eqref{3.25} holds for the sets $B_1'=B_1\cap (t_1+Q_L)\cap U_1$ and $B_2'=B_2\cap (t_1+Q_L)\cap V_1$. We may therefore assume that $I_{1,1}\leq -\eta^4/2^7$, but this assumption,
together with the previous assumption that
\bee I(t_1,t_2)= I_{1,1}+I_{1,2}+I_{2,1}+I_{2,2}\geq -\eta^4/2^9\eee
immediately implies that $I_{i,j}\geq\eta^4/2^9$ for some $(j,j')\neq (1,1)$ and \eqref{3.25} again follows.

It remains to consider the case when $I(t_1,t_2)\leq -\eta^4/2^9$ for all $(t_1,t_2)\in G_{\eta,\VE}$. Then by \eqref{3.27} and \eqref{3.30}
\begin{align*}
\iint I(t_1,t_2)\,dt_1\,dt_2 &=  \iint_{G_{\eta,\VE}} I(t_1,t_2)\,dt_1\,dt_2 + \iint_{G_{\eta,\VE}^c} I(t_1,t_2)\,dt_1\,dt_2
\leq -\frac{\eta^4}{2^4}\,\frac{\eta^4}{2^9}\,+\,2\,\frac{\eta^8}{2^{16}}\,
\leq\,-\frac{\eta^8}{2^{15}}.
\end{align*}
While on the other hand
\begin{align*}
\iint I(t_1,t_2)\,dt_1\,dt_2 %&=\frac{1}{\B_1\B_2}\iiiint
%f(x,y)\chi_L(x-t_1)\chi_L(y-t_2)\,dx\,dy\,dt_1\,dt_2
=O(L)
\end{align*}
by the first assumption of (\ref{assumption}),
which is a contradiction. This proves the theorem.
\end{proof}

\section{Proof of Proposition \ref{part11}}\label{SecPart11}

An appropriate ``relative generalized von-Neumann inequality'' will again be central to our proof of Proposition \ref{part11}, specifically a ``relative generalized von-Neumann inequality for product of simplices". 

However, the true heart of the argument is in fact the analogous result for \emph{just} simplices, the proof of this ``relative generalized von-Neumann inequality for simplices" is necessarily  significantly more involved than the analogous relative result for distances (whose proof was essentially identical to the non-relative case) and it is here that our loss in dimension appears.

We fix non-degenerate simplices $\Delta_{k_i}=\{v_0^i,v^i_1,v^i_2,\dots,v^i_{k_i}\}$ of dimension $k_i$ with $v_0^i=0$ and
\[c_{\Delta_{k_i}}=\min_{1\leq j\leq k_i}\text{dist}(v^i_j,\text{span}\left\{\{v^i_1,\dots,v^i_{k_i}\}\setminus v^i_j\right\})\leq1\]
and
let $B_i\subseteq[0,1]^{d_i}$ with $d_i\geq k_i+3$ and $\B_i=|B_i|>0$ denote arbitrary sets, for $i=1,2$. %Furthermore, on each of the sets $B_i$  we introduce the normalized measures
%\be\mu_{B_i}=\B_i^{-1}1_{B_i}.\ee

In contrast to the proof of Proposition \ref{part1}, we will need to assume that our sets $B_1$ and $B_2$ are suitably uniformly distributed, and make use of Proposition \ref{Propn00}, throughout the proof of Proposition \ref{part11}.

\subsection{A Relative Generalized von-Neumann Inequality for Simplices and Products of Simplices}

 \begin{defn}[Counting function for $\Delta_{k_1}\times \Delta_{k_2}$]
Let $0<\lm\ll1$.

 For functions $f_{ij}:[0,1]^{d_1}\times [0,1]^{d_2}\to\R$ with $(i,j)\in\{0,1,\dots,k_1\}\times\{0,1,\dots,k_2\}$  we define
%\[T_{B_1,B_2}(\lm):=T_{B_1,B_2}(f_{00},\dots,f_{k_1k_2})(\lm)\]
%where
\be
T_{\Delta_{k_1},\Delta_{k_2}}(f_{00},\dots,f_{k_1k_2})(\lm)=\iiiint \prod_{i=0}^{k_1}\prod_{j=0}^{k_2}f_{ij}(x-\lm\cdot U_1(v^1_i),y-\lm\cdot U_2(v^2_j))
\,d\mu_1(U_1)\,d\mu_2(U_2)\,dx\,dy\ee
\end{defn}

Note that if we let
\be
\widetilde{\nu}(x,y)=\nu_1(x)^{1/(k_2+1)}\nu_2(y)^{1/(k_1+1)}
\ee
where
$
\nu_1=\B_1^{-1}1_{B_1}$ and $\nu_2=\B_2^{-1}1_{B_2}
$
then
 \[T_{\Delta_{k_1},\Delta_{k_2}}(\widetilde{\nu},\dots,\widetilde{\nu})(\lm)=T_{\Delta_{k_1}}(\nu_1,\dots,\nu_1)(\lm)\cdot T_{\Delta_{k_2}}(\nu_2,\dots,\nu_2)(\lm)\] and  in light of Proposition \ref{Propn00} we can conclude that
\be\label{observation}
T_{\Delta_{k_1},\Delta_{k_2}}(\widetilde{\nu},\dots,\widetilde{\nu})(\lm)=1+O_{k_1,k_2}(\B_1^{-k_1-1}\B_2^{-k_2-1}c_{\Delta_{k_1}}^{-1/6}c_{\Delta_{k_2}}^{-1/6}\VE^{2/3})
\ee
for any $0<\lm\leq\VE\ll1$, provided $B_1$ and $B_2$ are $(\VE,\VE^4\lm)$-uniformly distributed subsets of $[0,1]^{d_1}$ and $[0,1]^{d_2}$.

\newpage

In this setting we have the following ``generalized von-Neumann inequality", for which it is essential that our count of product simplices is taken relative to suitably uniformly distributed sets $B_1$ and $B_2$.

\begin{lem}[Generalized von-Neumann for $\Delta_{k_1}\times \Delta_{k_2}$ relative to $B_1\times B_2$]\label{GvN11}

Let 
\[
\widetilde{\nu}=\nu_1^{1/(k_2+1)}\otimes\nu_2^{1/(k_1+1)}\quad\text{and}\quad \nu=\nu_1^{1/2}\otimes\nu_2^{1/2}
\]
where
$
\nu_1=\B_1^{-1}1_{B_1}$ and $\nu_2=\B_2^{-1}1_{B_2}
$
For any $0<\lm\leq\VE\ll\min\{c_{\Delta_{k_1}},c_{\Delta_{k_2}}\}$ and functions
\[f_{ij}:[0,1]^{d_1}\times [0,1]^{d_2}\to[-1,1]\]
with $(i,j)\in\{0,1,\dots,k_1\}\times\{0,1,\dots,k_2\}$ we have
\bee
|T_{\Delta_{k_1},\Delta_{k_2}}(f_{00}\widetilde{\nu},\dots,f_{k_1k_2}\widetilde{\nu})(\lm)|\leq \min\limits_{\substack{i=0,1,\dots,k_1\\ j=0,1,\dots,k_2}} \|f_{ij}\nu\|_{\Box(\VE^4\lm)} + O_{k_1,k_2}(\B_1^{-k_1-1}\B_2^{-k_2-1}c_{\Delta_{k_1}}^{-1/8}c_{\Delta_{k_2}}^{-1/8}\VE^{1/16})
\eee
provided $B_1$ and $B_2$ are $(\VE,\VE^4\lm)$-uniformly distributed subsets of $[0,1]^{d_1}$ and $[0,1]^{d_2}$ respectively.
\end{lem}

It is easy to see that Lemma \ref{GvN11}, combined with Proposition \ref{Propn00}, gives the following
\begin{cor}\label{CorGvN11}
Let $0<\A,\B_1,\B_2\leq 1$ and \[0<\lm\leq\VE\ll_{k_1,k_2}(c_{\Delta_{k_1}}c_{\Delta_{k_2}})^{2}(\B_1^{k_1+1}\B_2^{k_2+1}\A^{(k_1+1)(k_2+1)})^{16}.\]

If $A\subseteq B_1\times B_2\subseteq[0,1]^{d_1}\times[0,1]^{d_2}$ with $|A|=\A\B_1\B_2$ and %and $\B_i=|B_i|$ for $i=1,2$.
$
\|f_A\nu\|_{\Box(\VE^4\lm)}\ll\A^{(k_1+1)(k_2+1)},
$
then
%with $f_A=1_A-\A1_{B_1\times B_2}$
\bee
T_{\Delta_{k_1},\Delta_{k_2}}(1_A\widetilde{\nu},\dots,1_A\widetilde{\nu})(\lm)\geq\frac{1}{2}\A^{(k_1+1)(k_2+1)}
\eee
provided $B_1$ and $B_2$ are $(\VE,\VE^4\lm)$-uniformly distributed subsets of $[0,1]^{d_1}$ and $[0,1]^{d_2}$ respectively.
\end{cor}

\begin{proof}[Proof of Corollary \ref{CorGvN11}]
It follows immediately from Lemma \ref{GvN11} that
\begin{align*}
|T_{\Delta_{k_1},\Delta_{k_2}}(1_A\widetilde{\nu},\dots,1_A\widetilde{\nu})(\lm)-&T_{\Delta_{k_1},\Delta_{k_2}}(\A\widetilde{\nu},\dots,\A\widetilde{\nu})(\lm)|\\
&\leq (2^{(k_1+1)(k_2+1)}-1)\|f_A\nu\|_{\Box(\VE^4\lm)}+ O_{k_1,k_2}(\B_1^{-k_1-1}\B_2^{-k_2-1}c_{\Delta_{k_1}}^{-1/8}c_{\Delta_{k_2}}^{-1/8}\VE^{1/16})
\end{align*}
for any $0<\VE,\lm\ll\min\{c_{\Delta_{k_1}},c_{\Delta_{k_2}}\}$, where $
f_A=1_A-\A1_{B_1\times B_2}
$
while, as noted in (\ref{observation}), Proposition \ref{Propn00} implies that
\[
T_{\Delta_{k_1},\Delta_{k_2}}(\A\widetilde{\nu},\dots,\A\widetilde{\nu})(\lm)=
%\A^{(k_1+1)(k_2+1)}T_{\Delta_{k_1},\Delta_{k_2}}(\widetilde{\nu},\dots,\widetilde{\nu})(\lm)=
\A^{(k_1+1)(k_2+1)}(1+O_{k_1,k_2}(\B_1^{-k_1-1}\B_2^{-k_2-1}c_{\Delta_{k_1}}^{-1/6}c_{\Delta_{k_2}}^{-1/6}\VE^{2/3}))
\]
for any $0<\lm\leq\VE\ll1$, as required.
\end{proof}

\subsection{A Relative Version of Lemma \ref{GvN00}}\label{6.2}
%Let $\Delta_{k}=\{0,v_1,\dots,v_{k}\}$ be any non-degenerate $k$-dimensional simplex and  $B\subseteq[0,1]^{d}$ with $d\geq k+1$ be an arbitrary set with $\B=|B|>0$.

Key to the proof of Lemma \ref{GvN11} is the following

\begin{lem}[Lemma \ref{GvN00} relative to uniformly distributed sets]\label{relative}
Let $\Delta_k=\{0,v_1,v_2,\dots,v_{k}\}$ be any non-degenerate $k$-dimensional simplex with \[c_{\Delta_k}=\min_{1\leq j\leq k}\text{\emph{dist}}(v_j,\text{\emph{span}}\left\{\{v_1,\dots,v_{k}\}\setminus v_j\right\})\leq 1\] and
$B\subseteq[0,1]^{d}$ with $d\geq k+3$ be an arbitrary set with $\B=|B|>0$.
If we set $\nu=\B^{-1}1_B$, then for any $0<\lm\leq\VE\ll c_{\Delta_k}$ and functions
$f_0,f_1,\dots,f_{k}:[0,1]^d\to[-1,1]$ we have
\be\label{pp}
\left|T_{\Delta_k}(f_0\nu,\dots,f_{k}\nu)(\lm)\right|^2\leq\iint f_j\nu(x)f_j\nu (x-x_1)\psi_{\VE^4\lm}(x_1)\,dx\,dx_1+O_{k}(\B^{-3k-3}c_{\Delta_k}^{-1/2}\VE^{1/4})
\ee
for any $0\leq j\leq k$, provided $B$ is a $(\VE,\VE^4\lm)$-uniformly distributed subset of $[0,1]^{d}$.
\end{lem}

\begin{proof}
%We may assume that $d=k+3$.
As in the proof of Lemma \ref{GvN00} it suffices, by symmetry, to establish (\ref{pp}) for $j=k$.
Note also, as in (\ref{observation}) above, that Proposition \ref{Propn00} implies
\be\label{observation2}
T_{\Delta_k}(\nu,\dots,\nu)(\lm)=1+O_k(\B^{-k-1}c_{\Delta_k}^{-1/6}\VE^{2/3}),
\ee
provided $0<\lm\leq\VE\ll1$ and $B$ is an $(\VE,\VE^4\lm)$-uniformly distributed subset of $[0,1]^{d}$ with $d\geq k+1$. It is equally easy to see, using Lemma \ref{GvN00}, that if $1\leq j\leq k$ and any $j$ of the weights $\nu$ are replaced with $1_{[0,1]^d}$ then this modified count will still be asymptotically equal to $1$ and will in fact equal $1+O_k(\B^{-k-1+j}c_{\Delta_k}^{-1/6}\VE^{2/3})$.

%We will write
%\[T_{\Delta_k}(\lm)=T_{\Delta_k}(f_0\nu,\dots,f_{k}\nu)(\lm)\]
Since
\begin{align*}
\left|T_{\Delta_k}(f_0\nu,\dots,f_{k}\nu)(\lm)\right|\leq\iint\cdots\int \nu(x)\nu(x-\lm x_{1})\cdots \nu(x-\lm x_{k-1})\Bigl|  \int f_k\nu & (x-\lm x_k) \,d\sigma^{(d-k)}_{x_1,\dots,x_{k-1}}(x_k) \Bigr|\\
&\,d\sigma^{(d-k+1)}_{x_1,\dots,x_{k-2}}(x_{k-1})\cdots d\sigma(x_{1})\,dx
\end{align*}
it follows from  an application of Cauchy-Schwarz, facilitated by (\ref{observation2}) for the simplex $\Delta_{k-1}$, that
\bee
|T_{\Delta_k}(f_0\nu,\dots,f_{k}\nu)(\lm)|^2\leq\bigl(1+O_{k}(\B^{-k}c_{\Delta_k}^{-1/6}\VE^{2/3})\bigr)^2
\bigl(\,M(\lm)+E(\lm)\,\bigr)
%\iint\cdots\int \nu(x)\nu(x-\lm x_{1})\cdots \nu(x-\lm x_{k-1}) \Bigl|  \int  f_k \nu (x-\lm x_k)\,d\sigma^{(d-k)}_{x_1,\dots,x_{k-1}}(x_k) \Bigr|^2\\
%&\quad\quad\quad\quad\quad\quad\quad\quad\quad\quad\quad\quad\quad\quad\quad\quad\quad\quad\quad\quad\quad\quad\,d\sigma^{(d-k+1)}_{x_1,\dots,x_{k-2}}(x_{k-1})\cdots d\sigma(x_{1})\,dx\\
%&= M(\lm)+E(\lm)+ ERROR?
\eee
where
\bee
M(\lm)=\iint\cdots\int  \Bigl|  \int  f_k \nu (x-\lm x_k)\,d\sigma^{(d-k)}_{x_1,\dots,x_{k-1}}(x_k) \Bigr|^2\,d\sigma^{(d-k+1)}_{x_1,\dots,x_{k-2}}(x_{k-1})\cdots d\sigma(x_{1})\,dx
\eee
and
\begin{align*}
E(\lm)&=\iint\cdots\int \bigl[\nu(x)\nu(x-\lm x_{1})\cdots \nu(x-\lm x_{k-1})-1(x)\bigr] \Bigl|  \int  f_k \nu (x-\lm x_k)\,d\sigma^{(d-k)}_{x_1,\dots,x_{k-1}}(x_k) \Bigr|^2\\
&\quad\quad\quad\quad\quad\quad\quad\quad\quad\quad\quad\quad\quad\quad\quad\quad\quad\quad\quad\quad\quad\quad\quad\quad\quad\quad\quad\quad\,d\sigma^{(d-k+1)}_{x_1,\dots,x_{k-2}}(x_{k-1})\cdots d\sigma(x_{1})\,dx
\end{align*}
where $1=1_{[0,1]^d}$.

It follows from the proof of Lemma \ref{GvN00}, specifically the argument from  (\ref{CS})  to  (\ref{72})) that
\bee
M(\lm)\leq \iint f_k\nu(x_1)f_k\nu(x_2)\psi_{\VE^4\lm}(x_2-x_1)\,dx_1\,dx_2.
\eee

We now complete the proof by establishing that  $E(\lm)=O_k(\B^{-k-3}c_{\Delta_k}^{-1/6}\VE^{1/4})$. Our strategy will be to expand the square in the error term $E(\lm)$ which will add a new vertex $x_{k+1}$ to the simplex. ``Fixing" the distance $|x_{k+1}-x_k|$ leads to an expression which may be viewed as the difference between a weighted and an unweighted average over all isometric copies of a fixed $(k+1)$-dimensional simplex. The reason that this difference is small is that the measure $\nu$ behaves suitably random with respect to averages of this type, expressed in (\ref{observation2}). To remove the uncontrolled terms $f_k$ one needs another application of Cauchy-Schwarz which leads to simplices of dimension $k+2$ and the requirement $d\geq k+3$ for the underlying dimension of the space.

Writing
\begin{align*}
\nu(x)\nu(x-\lm x_{1})\cdots \nu(x-\lm x_{k-1})-1(x)&=\sum_{j=0}^{k-1}\bigl[\nu(x-\lm x_{j})-1(x)\bigr]\nu(x-\lm x_{j+1})\cdots \nu(x-\lm x_{k-1})
%&=\sum_{j=0}^{k-1}\bigl[\nu(x-\lm x_{j})-1(x-\lm x_{j})\bigr]\nu(x-\lm x_{j+1})\cdots \nu(x-\lm x_{k-1})
\end{align*}
with the understanding that $x_0=0$, it follows that
\bee
E(\lm)=\sum_{j=0}^{k-1}E_j(\lm)
\eee
with
\begin{align*}
E_j(\lm)&=\iint\cdots\int \bigl[\nu(x-\lm x_{j})-1(x)\bigr]\nu(x-\lm x_{j+1})\cdots \nu(x-\lm x_{k-1}) \Bigl|  \int  f_k \nu (x-\lm x_k)\,d\sigma^{(d-k)}_{x_1,\dots,x_{k-1}}(x_k) \Bigr|^2\\
&\quad\quad\quad\quad\quad\quad\quad\quad\quad\quad\quad\quad\quad\quad\quad\quad\quad\quad\quad\quad\quad\quad\quad\quad\quad\quad\quad\quad\quad\,d\sigma^{(d-k+1)}_{x_1,\dots,x_{k-2}}(x_{k-1})\cdots d\sigma(x_{1})\,dx.
\end{align*}
Squaring out we see that
\begin{align*}
E_j(\lm)=&\iint\cdots\iiint \bigl[\nu(x-\lm x_{j})-1(x)\bigr]\nu(x-\lm x_{j+1})\cdots \nu(x-\lm x_{k-1})   f_k \nu (x-\lm x_k)f_k \nu (x-\lm x_{k+1})\\
&\quad\quad\quad\quad\quad\quad\quad\quad\quad\quad\quad\quad\quad
\,d\sigma^{(d-k)}_{x_1,\dots,x_{k-1}}(x_{k+1})\,d\sigma^{(d-k)}_{x_1,\dots,x_{k-1}}(x_{k})
\,d\sigma^{(d-k+1)}_{x_1,\dots,x_{k-2}}(x_{k-1})\cdots d\sigma(x_{1})\,dx.
\end{align*}

Since $d\geq k+3$
we can follow the argument in Section \ref{direct} and write
\bee
\sigma^{(d-k)}_{x_1,\dots,x_{k-1}}(x_{k+1})=\int_0^{\pi} (\sin\theta_1)^{d-k-1}\,d\sigma^{(d-k-1)}_{x_1,\dots,x_{k-1},x_k,\theta_1}(x_{k+1})\,d\theta_1
\eee
where $\sigma^{(d-k-1)}_{x_1,\dots,x_{k-1},x_k,\theta_1}(x_{k+1})$ denotes the normalized measure on the sphere
\bee
S^{d-k-1}_{x_1,\dots,x_{k-1},x_k,\theta_1}=S^{d-1}(0,|v_{k+1}|)\cap S^{d-1}(x_1,|v_{k+1}-v_1|)\cap\cdots\cap S^{d-1}(x_{k},|v_{k+1}-v_{k}|)\eee
with $v_{k+1}=v_{k+1}(\theta)$ satisfying
$
|v_{k+1}|=|v_k|,
$
$
|v_{k+1}-v_j|=|v_k-v_j|
$
for all $1\leq j\leq k-1$ and $\theta_1$ determining the angle between $v_{k+1}$ and $v_k$ so that
$ |v_{k+1}-v_k|=2|v_k|\sin(\theta_1/2).$

If we now let $\Delta_{k+1}(\theta_1)=\{0,v_1,\dots,v_k,v_{k+1}\}$, then it follows
(again using (\ref{observation2})) that
\bee
E_j(\lm)=\int_0^\pi (\sin\theta_1)^{d-k-1}T_{\Delta_{k+1}(\theta)}(1,\dots,1,\nu-1,\nu,\dots,\nu,f_{k}\nu,f_{k}\nu)(\lm)\,d\theta_1+O_k(\B^{-k-1+j}c_{\Delta_k}^{-1/6}\VE^{2/3})
\eee
where $T_{\Delta_{k+1}(\theta)}(1,\dots,1,\nu-1,\nu,\dots,\nu,f_{k}\nu,f_{k}\nu)(\lm)$ equals
\begin{align*}
&\iint\cdots\iiint \bigl[\nu(x-\lm x_{j})-1\bigr]\nu(x-\lm x_{j+1})\cdots \nu(x-\lm x_{k-1})   f_k \nu (x-\lm x_k)f_k \nu (x-\lm x_{k+1})\\
&\quad\quad\quad\quad\quad\quad\quad\quad\quad\quad\quad\quad\quad
\,d\sigma^{(d-k-1)}_{x_1,\dots,x_{k-1},x_k,\theta_1}(x_{k+1})\,d\sigma^{(d-k)}_{x_1,\dots,x_{k-1}}(x_{k})
\,d\sigma^{(d-k+1)}_{x_1,\dots,x_{k-2}}(x_{k-1})\cdots d\sigma(x_{1})\,dx.
\end{align*}

In light of (\ref{i}) and (\ref{ii}) it suffices to now show that
\bee
%E_j(\lm)=\int_0^\pi (\sin\theta)^{d-k-1}T_{\Delta_{k+1}(\theta)}(1_{[0,1]^d},\dots,1_{[0,1]^d}, \nu-1,\nu,\dots\nu,f_{k},f_{k})(\lm)\,d\theta + O(\lm)
E_j'(\lm):=  \int_0^\pi (\sin\theta_1)^{d-k-1}T_{\Delta'_{k+1-j}(\theta_1)}(f_{k}\nu,f_{k}\nu,\nu,\dots\nu, \nu-1)(\lm)\,d\theta_1=O(\B^{-k-1+j}\VE^{1/4})\eee
where
\bee
\Delta'_{k+1-j}(\theta_1)=\{0,v_1',\dots,v'_{k+1-j}\}
\eee
with $v'_i=v_{k+1-i}-v_{k+1}$ for $0\leq i\leq k$ and $v'_{k+1}=-v_{k+1}$.

Since %$|T_{\Delta'_{k+1-j}(\theta)}(\lm)|:=
$|T_{\Delta'_{k+1-j}(\theta_1)}(f_{k}\nu,f_{k}\nu,\nu,\dots\nu, \nu-1)(\lm)|$ is dominated by
\begin{align*}
%\left|T_{\Delta'_{k+1-j}(\theta)}(f_{k}\nu,f_{k}\nu,\nu,\dots\nu, \nu-1)(\lm)\right|\leq
\iint\cdots\int \nu(x)\nu(x-\lm x_{1})\cdots \nu(x-\lm x_{k-j})\Bigl|  \int (\nu-1)(x-\lm & x_{k+1-j})  \,d\sigma'^{(d-k-1+j)}_{x_1,\dots,x_{k-j}}(x_{k+1-j}) \Bigr|\\
&\,d\sigma'^{(d-k+j)}_{x_1,\dots,x_{k-j-1}}(x_{k-j})\cdots d\sigma'(x_{1})\,dx
\end{align*}
it follows from  an application of Cauchy-Schwarz, facilitated by (\ref{observation2}) for the simplex $\Delta'_{k-j}(\theta_1)$, that
\bee
|T_{\Delta'_{k+1-j}(\theta_1)}(f_{k}\nu,f_{k}\nu,\nu,\dots,\nu, \nu-1)(\lm)|^2\leq %\bigl(1+O_{\Delta_k}(\B^{-k-1+j}\VE^{2/3})\bigr)^2
2\,
 I_{\Delta'_{k+1-j}(\theta_1)}(\lm)
\eee
where
\begin{align*}
%\left|T_{\Delta'_{k+1-j}(\theta)}(f_{k}\nu,f_{k}\nu,\nu,\dots\nu, \nu-1)(\lm)\right|\leq
I_{\Delta'_{k+1-j}(\theta_1)}(\lm)=\iint\cdots\int \nu(x)\nu(x-\lm x_{1})\cdots \nu(x-\lm x_{k-j})\Bigl|  \int (\nu-1)(x-\lm & x_{k+1-j}) \,d\sigma'^{(d-k-1+j)}_{x_1,\dots,x_{k-j}}(x_{k+1-j}) \Bigr|^2\\
&\,d\sigma'^{(d-k+j)}_{x_1,\dots,x_{k-j-1}}(x_{k-j})\cdots d\sigma'(x_{1})\,dx.
\end{align*}
Squaring out we see that $I_{\Delta'_{k+1-j}(\theta_1)}(\lm)$ equals
\begin{align*}
%\left|T_{\Delta'_{k+1-j}(\theta)}(f_{k}\nu,f_{k}\nu,\nu,\dots\nu, \nu-1)(\lm)\right|\leq
\iint\cdots&\iiint \nu(x)\nu(x-\lm x_{1})\cdots \nu(x-\lm x_{k-j})(\nu-1)(x-\lm  x_{k+1-j})\,(\nu-1)(x-\lm  x_{k+2-j})\\
&\quad\quad\quad\quad\quad\quad\quad\quad\quad\quad \,d\sigma'^{(d-k-1+j)}_{x_1,\dots,x_{k-j}}(x_{k+2-j})
 \,d\sigma'^{(d-k-1+j)}_{x_1,\dots,x_{k-j}}(x_{k+1-j})
\,d\sigma'^{(d-k+j)}_{x_1,\dots,x_{k-j-1}}(x_{k-j})\cdots d\sigma'(x_{1})\,dx.
\end{align*}

Since $d\geq k+3$ we can again argue as above to obtain
\begin{align*}
%\left|T_{\Delta'_{k+1-j}(\theta)}(f_{k}\nu,f_{k}\nu,\nu,\dots\nu, \nu-1)(\lm)\right|\leq
I_{\Delta'_{k+1-j}(\theta_1)}(\lm)&=
\int_0^\pi (\sin\theta_2)^{d-k-2+j}\,\bigl[T_1(\lm)-T_2(\lm)-T_3(\lm)+T_4(\lm)\bigr]\,d\theta_2
\end{align*}
where
\[T_1(\lm)=T_{\Delta'_{k+2-j}(\theta_1,\theta_2)}(\nu,\dots,\nu)(\lm)\]
\[T_2(\lm)=T_{\Delta'_{k+2-j}(\theta_1,\theta_2)}(\nu,\dots,\nu, 1)(\lm)\]
\[T_3(\lm)=T_{\Delta'_{k+2-j}(\theta_1,\theta_2)}(\nu,\dots,1,\nu)(\lm)\]
\[T_4(\lm)=T_{\Delta'_{k+2-j}(\theta_1,\theta_2)}(\nu,\dots,\nu, 1,1)(\lm)\]
with
\bee
\Delta'_{k+2-j}(\theta_1,\theta_2)=\Delta'_{k+1-j}(\theta_1)\cup\{v'_{k+2-j}\}
\eee
with $v'_{k+2-j}=v'_{k+2-j}(\theta_2)$ satisfying
$
|v'_{k+2-j}|=|v'_{k+1-j}|,
$
$
|v'_{k+2-j}-v_i|=|v'_{k+1-j}-v_i|
$
for all $1\leq i\leq k-j$ and $\theta_2$ determining the angle between $v'_{k+2-j}$ and $v'_{k+1-j}$ so that
$ |v'_{k+2-j}-v'_{k+1-j}|=2|v_{j}|\sin(\theta_2/2).$

We have therefore ultimately established
\bee
|E_j'(\lm)|^2\leq C \int_0^\pi\int_0^\pi
%(\sin\theta_1)^{d-k-1}(\sin\theta_2)^{d-k-2+j}\,
\bigl|T_1(\lm)-T_2(\lm)-T_3(\lm)+T_4(\lm)\bigr|\,d\theta_2\,d\theta_1
\eee
for each $0\leq j\leq k$. In light of (\ref{observation2}) we know that
\bee
T_i(\lm)=1+O_k(\B^{-k-3+j}c_{\Delta'_{k+2-j}(\theta_1,\theta_2)}^{-1/6}\VE^{2/3})
\eee
for $i=1,\dots,4$, and hence
\bee
E_j'(\lm)=O_k(\B^{-k-3+j}\VE^{1/4})
\eee
provided
\bee
c_{\Delta'_{k+2-j}(\theta_1,\theta_2)}\geq\VE.
\eee
The result now follows since the fact that
\bee
c_{\Delta'_{k+2-j}(\theta_1,\theta_2)}= \min\{c_{\Delta_{k}}, 2|v_k|\sin(\theta_1/2), 2|v_j|\sin(\theta_2/2)\}
\eee
and $\VE\ll c_{\Delta_{k}}$ ensures that
\bee
|\{(\theta_1,\theta_2)\in[0,\pi]\times[0,\pi]\,:\, c_{\Delta'_{k+2-j}(\theta_1,\theta_2)}\leq\VE\}|=O(\VE).\qedhere
\eee
\end{proof}

\subsection{Proof of Lemma \ref{GvN11}}

The proof of Lemma \ref{GvN11} will follow from two applications of Cauchy-Schwarz combined with Proposition \ref{Propn00} and  Lemma \ref{relative}.
We first observe that if
\[
T_{\Delta_{k_1},\Delta_{k_2}}(\lm):=T_{\Delta_{k_1},\Delta_{k_2}}(f_{00}\widetilde{\nu},\dots,f_{k_1k_2}\widetilde{\nu})(\lm)
\]
then
\[
T_{\Delta_{k_1},\Delta_{k_2}}(\lm)=\iint
\nu_1(x-\lm U_1(v^1_0))\cdots\nu_1(x-\lm U_1(v^1_{k_1}))
\,T_{\Delta_{k_2}}(g_0^{x,U_1}\nu_2,g_1^{x,U_1},\dots,g_{k_2}^{x,U_1}\nu_2)(\lm)\,d\mu_1(U_1)\,dx
\]
where
\bee
g_j^{x,U_1}(y)=f_{0j}(x-\lm\cdot U_1(v^1_0),y)\cdots f_{k_1j}(x-\lm\cdot U_1(v^1_{k_1}),y)
%\prod_{i=0}^{k_1}f_{ij}(x+\lm\cdot U_1(v^1_i),y)
\eee
 for each $j=0,1,\dots,k_2$ and that Lemma \ref{relative} implies
\[
\left|T_{\Delta_{k_2}}(g_0^{x,U_1}\nu_2,\dots,g_{k_2}^{x,U_1}\nu_2)(\lm)\right|^2
\leq\iint g_{j}^{x,U_1}\nu_2(y)g_{j}^{x,U_1}\nu_2(y-y_1)\psi_{2,\VE^4\lm}(y_1)\,dy\,dy_1+O_{k_2}(\B^{-3k_2-3}c_{\Delta_{k_2}}^{-1/2}\VE^{1/4})
\]
for any $0\leq j\leq k_2$.
Hence by Cauchy-Schwarz, using (\ref{observation2}) for $T_{\Delta_{k_1}}(\nu_1,\dots,\nu_1)(\lm)$, and switching the order of integration we obtain that $|T_{\Delta_{k_1},\Delta_{k_2}}(\lm)|^2$ is majorized by
\[
%\left|T_{\Delta_{k_1},\Delta_{k_2}}(\lm)\right|^2\leq 
\iint T_{\Delta_{k_1}}(h_{0j}^{y,y_1}\nu_1,\dots,h_{k_1j}^{y,y_1}\nu_1)(\lm)\,\nu_2(y)\nu_2(y-y_1)\psi_{2,\VE^4\lm}(y_1)\,dy\,dy_1+O_{k_1,k_2}(\B_1^{-k_1-1}\B_2^{-3k_2-3}c_{\Delta_{k_1}}^{-1/6}c_{\Delta_{k_2}}^{-1/2}\VE^{1/4})
\]
for any $0\leq j\leq k_2$ where
\bee
h_{ij}^{y,y_1}(x)=f_{ij}(x,y)f_{ij}(x,y-y_1)
\eee
for $i=0,1,\dots,k_1$.
A further application of Cauchy-Schwarz (using the fact that $\psi_{2,\VE^4\lm}$ is $L^1$-normalized) and appeal to  Lemma \ref{relative} reveals that $|T_{\Delta_{k_1},\Delta_{k_2}}(\lm)|^4$ is majorized by
\begin{align*}
\iiiint h_{ij}^{y,y_1}\nu_1(x)h_{ij}^{y,y_1}\nu_1(x-x_1)
\,\nu_2(y)\nu_2(y-y_1)\psi_{1,\VE^4\lm}(x_1)\,&\psi_{2,\VE^4\lm}(y_1)\,dx\,dx_1\,dy\,dy_1\\
&+O_{k_1,k_2}(\B_1^{-4k_1-4}\B_2^{-3k_2-4}c_{\Delta_{k_1}}^{-1/2}c_{\Delta_{k_2}}^{-1/2}\VE^{1/4})
\end{align*}
for any $0\leq i\leq k_1$ and $0\leq j\leq k_2$.
Since
\[
h_{ij}^{y,y_1}\nu_1(x)h_{ij}^{y,y_1}\nu_2(x-x_1)\nu_2(y)\nu_2(y-y_1)=f_{ij}\nu(x,y)f_{ij}\nu(x-x_1,y)f_{ij}\nu(x,y-y_1)f_{ij}\nu(x-x_1,y-y_1).
\]
the result follows from (\ref{almost2}).
\qed

\subsection{Inverse Theorem Revisited}

We complete this section by noting the following immediate consequence of Theorem \ref{InvThm} which together with Corollary \ref{CorGvN11} implies Proposition \ref{part11}.

\begin{cor}\label{InvCor11}
Let $0<\A,\B_1,\B_2\leq 1$ and $B_1$ and $B_2$ be $(\VE,\VE^4\lm)$-uniformly distributed subsets of $[0,1]^{d_1}$ and $[0,1]^{d_2}$  with $0<\lm\leq\VE\ll\B_1^{k_1+1}\B_2^{k_2+1}\A^{8(k_1+1)(k_2+1)}$.
If $A\subseteq B_1\times B_2\subseteq[0,1]^{d_1}\times[0,1]^{d_2}$ with $|A|=\A\B_1\B_2$ and
\[\|f_A\nu\|_{\Box(\VE^4\lm)}\gg\A^{(k_1+1)(k_2+1)}\]
with $f_A=1_A-\A1_{B_1\times B_2}$,
then there exist cubes $Q_i\subseteq[0,1]^{d_i}$ of side-length $\VE^4\lm$ and sets $B_i'$  in $Q_i$ for which
\bee
\frac{|A\cap(B_1'\times B_2')|}{|B_1'\times B_2'|}\geq \A+c\,\A^{8(k_1+1)(k_2+1)}.
\eee
\end{cor}

%\bigskip

\section{Proof of Proposition \ref{Propn1}, Part II: Regularization}\label{SecPart2}

To complete the proof of Proposition \ref{Propn1}, as was noted after the Proposition \ref{part1},
we need to now produce a pair of new sets $B_1''$ and $B_2''$ that are $(\eta,L')$-uniformly distributed for a sufficiently small $\eta$ and for $L'$ attached to some of the $\lm_j$'s, but for which $A$ still has increased density on $B_1''\times B_2''$. Proposition \ref{part1} did produce a pair of sets $B_1'$ and $B_2'$ for which $A$ has increased density on $B_1'\times B_2'$, but these sets are not necessarily uniformly distributed. We will now obtain sets $B_1''$ and $B_2''$ with the desired properties from the sets $B_1$ and $B_2$ produced by Proposition \ref{part1} by appealing to a version of Szemer\'edi's Regularity Lemma \cite{Sz}  adapted to a sequence of scales $\{L_j\}_{1\leq j\leq J}$.

The precise result we need is stated below in Theorem \ref{reglemma},
but first we state a couple of definitions.

%Let $B_i\subseteq[0,1]^2$ with $|B_i|=\Be_i$  for $i=1,2$ and
\begin{defn}[A partition $\PP$ being adapted to scale $L_j$]
Let $1=L_0> L_1>L_2>\cdots >0$ be a sequence with the property that $L_{j+1}<\frac{1}{2} L_j$.
We say that a partition $\PP=\QQ\cup\RR$ of $[0,1]^{d_1}\times [0,1]^{d_2}$ into cubes $\QQ$
%=Q^i_1\times Q^i_2$
and ``rectangles" $\RR$ is \emph{adapted to the scale $L_j$} if each of the cubes in $\QQ$ have sidelength $L_i$ for some $0\leq i\leq j$.
\end{defn}

\begin{defn}[$(\VE,L)$-uniform distribution on $Q$]
Let $Q$ be a cube of sidelength $L_0$ and
$0<L/L_0\leq\VE\ll1$.

A set $B\subseteq Q$ is said to be $(\VE,L)$-uniformly distributed on $Q$ if
\be\label{3.3.1}
\frac{1}{|Q|}\int_{Q}\left|\frac{|B\cap(t+Q_{L})|}{|Q_{L}|}-\frac{|B|}{|Q|}\,\right|^2\,dt\leq\VE^2.
\ee
\end{defn}

\begin{thm}[Regularity Lemma]\label{reglemma} Let $0<\Be_1,\Be_2,\eta\leq1$ and $B_i\subseteq[0,1]^{d_i}$ with $|B_i|=\Be_i$ for $i=1,2$.

Given any sequence $1=L_0> L_1>\cdots >0$ with $L_{j+1}<\frac{1}{2} L_j$ there exists $0\leq j<j'\leq J(\B_1,\B_2,\eta)$ and a partition $\PP=\QQ\cup\RR$ of $[0,1]^{d_1}\times [0,1]^{d_2}$ adapted to the scale $L_j$ with the following properties:
\begin{itemize}
\item[(i)]  For every cube $Q=Q_1\times Q_2$ in $\QQ$ of sidelength $L_i$ with $0\leq i\leq j-1$, the sets $B_1$ and $B_2$ are $(\eta,L_{j'})$-uniformly distributed on the cubes $Q_1$ and $Q_2$ respectively.\\

\item[(ii)] If $\NN$ denotes the collection of cubes in $Q=Q_1\times Q_2$ in $\QQ$ of sidelength $L_j$ for which at least one of the sets $B_1$ and $B_2$ is \emph{not} $(\eta,L_{j'})$-uniformly distributed on the cubes $Q_1$ and $Q_2$ respectively, then
\[\sum_{Q\in \NN} |Q|+ \sum_{R\in \RR} |R|\leq\eta.\]
\end{itemize}
\end{thm}

%Note that the partition $\PP$ defines a partition of the set $B=B_1\times B_2$ such that $B\cap Q$ is $(\eta,L_{j'})$-uniformly distributed on $Q$ for almost every cell $Q\in \PP$, we may refer to such a partition $\eta$-\emph{regular}. From this it is easy to obtain the following.

The proof of Theorem \ref{reglemma} follows by standard arguments, for completeness we include it in Section \ref{reglemproof}.\\

An almost immediate consequence of Theorem \ref{reglemma} is the following Corollary which, together with Proposition \ref{part1}, provides a complete proof of Proposition \ref{Propn1}, the easy verification of this we leave to the reader.

\begin{cor}\label{cor3.3}
Let $0<\A,\B_1,\B_2,\tau,\eps\leq1$ and $A\subseteq B_1\times B_2\subseteq[0,1]^{d_1}\times [0,1]^{d_2}$ with $|A|\geq (\al+\tau)\B_1\B_2$ and $|B_i|=\Be_i$ for $i=1,2$.
Given any sequence $1=L_0> L_1>\cdots >0$ with $L_{j+1}<\frac{1}{2} L_j$, there exist $0\leq j<j'\leq J(\A,\B_1,\B_2,\tau,\eps)$ and squares $Q_1,\,Q_2$ of sidelength $L_j$ such that the sets
\[B_i':=B_i\cap Q_i\]
with $i=1,2$ have the following properties:
\begin{itemize}
\item[(i)]
 $|B'_i|\geq \dfrac{1}{3}\Be_i\tau|Q_i|$.\\

\item[(ii)] $B_i'$ is $(\eps, L_{j'})$-uniformly distributed on $Q_i$\\

\item[(iii)] $\dfrac{|A\cap(B'_1\times B'_2)|}{|B'_1\times B'_2|}\geq \al+\dfrac{\tau}{3}$.
\end{itemize}
\end{cor}

\begin{proof}[Proof that Theorem \ref{reglemma} implies Corollary \ref{cor3.3}]

Let $\eta=\eps\Be_1\Be_2\tau/3$ and $\PP=\QQ\cup\RR$ be a partition of $[0,1]^{d_1}\times [0,1]^{d_2}$ adapted to the scale $L_j$ that satisfies the conclusions of Theorem \ref{reglemma} for some $0\leq j<j'\leq J(\B_1,\B_2,\eta)$.

%Let $B=B_1\times B_2$, $\Be=\Be_1\Be_2$ and $\{C_k\}_{k=1}^K$ denote the collection of cells that constitute $\PP$. For each cell $C_k$ we let $\de_k$ %=\dfrac{|B\cap C_k|}{|C_k|}$
%denote the relative density $B$ in $C_k$ and $\al_k$ %=\dfrac{|A\cap(B\cap C_k)|}{|B\cap C_k|}$
%denote the relative density of $A$ on $B\cap C_k$.
%Since $A\subs B$ we have by our assumption on the relative density of $A$ on $B$ that
%\be\label{3.4.4}
%|A|=|A\cap B|=\de(A|B)\,|B|\geq (\al+\tau)\Be.
%\ee

Let $B=B_1\times B_2$ and $\UU$ denote the collection of all cubes in $Q=Q_1\times Q_2$ in $\QQ$ of sidelength $L_i$ with $0\leq i\leq j$ for which $B_1$ and $B_2$ are $(\eta,L_{j'})$-uniformly distributed on $Q_1$ and $Q_2$ respectively. Note that property (ii) of Corollary \ref{cor3.3}  holds by definition for all cubes $Q_1$ and $Q_2$ for which $Q_1\times Q_2\in\UU$.

If we let $\mS$ denote the collection of all cubes $Q$ in $\UU$ which are \emph{sparse} in the sense that $|B\cap Q|<\Be\tau|Q|/3$, then property (i) of Corollary \ref{cor3.3}  will hold by definition for all cubes $Q_1$ and $Q_2$ with $Q_1\times Q_2\in\UU\setminus\mS$.
Finally,
 it is straightforward to  see, using property (ii) of our partition $\PP$ (on the size of $\NN$ and $\RR$) and our assumption on the relative density of $A$ on $B$,  that property (iii) of Corollary \ref{cor3.3} must hold for at least one cube $Q$ in $\UU\setminus\mS$. %Now assume indirectly that for all cubes $Q$ in $\UU\setminus\mS$ we have that $\de(A|B\cap C^i)\leq (\al+\tau/8)$. Note that $B\cap C^i=(B_1\cap Q_1)\times (B_2\cap Q_2)$ so this means that property (3) does not hold for any cube in $\VV^{(j)}$. Then we have
%\begin{align*}
%|A| &= \sum_{i\in \PP^{j}} |A\cap B\cap C^i| = \sum_{i\in\PP^{(j)}} \al^i\,\de^i\,|C^i|\\
%&\leq \sum_{i\in\VV^{(j)}} \al^i\,\de^i\,|C^i| + \sum_{i\in\NN^{(j)}} \al^i\,\de^i\,|C^i| + \sum_{i\in\mS^{(j)}} \al^i\,\de^i\,|C^i| + \frac{\Be\tau}{8}\\
%&\leq \sum_{i\in\VV^{(j)}} (\al+\frac{\tau}{8})\, \de^i\,|C^i| + \sum_{i\in\NN^{(j)}} |C^i| + \sum_{i\in\mS^{(j)}} \frac{\Be\tau}{8}\,|C^i|
%     +  \frac{\Be\tau}{8}\\
%     &\leq (\al+\frac{\tau}{8}) \Be + \frac{\Be\tau}{4} + \frac{\Be\tau}{8} + \frac{\Be\tau}{8} < (\al+\tau)\Be.
%\end{align*}
%This contradicts \eqref{3.4.4} and hence the must be a cell $C=Q_1\times Q_2 \in \VV^{(j)}$ for which $\de(A|B\cap C^i)\geq (\al+\tau/8)$ and then conclusion (3) holds.
\end{proof}

\subsection{Proof of Theorem \ref{reglemma}}\label{reglemproof}
%\begin{proof}[Proof of Theorem \ref{reglemma}]

By passing to a subsequence we may assume $L_{j+1}\leq 2^{-(j+6)}\eta L_j$, and in this case we will show that the conclusions of the theorem hold with $j'=j+1$ for some $0\leq j\leq J(\B_1,\B_2,\ga,\eta)$.

For $j=0,1,2,\dots$ we construct partitions $\PP^{(j)}$ of $[0,1]^{d_1}\times [0,1]^{d_2}$ into cubes $\QQ^{(j)}$ and rectangles $\RR^{(j)}$ starting from the trivial partition $\PP^{(0)}$ consisting of only one cube $Q=[0,1]^{d_1}\times [0,1]^{d_2}$. The partition $\PP^{(j)}$ will consists of two collections of cubes $\UU^{(j)},\NN^{(j)}$ and a collection of rectangles $\RR^{(j)}$, that is
\[\PP^{(j)}=\UU^{(j)}\cup\NN^{(j)}\cup\RR^{(j)}.\]
The collection $\RR^{(j)}$ will consist of rectangles $R=R_1\times R_2$ whose total measure is small, specifically
\be\label{3.4.2}
\sum_{R\in \RR^{(j)}} |R|\leq \frac{\eta}{2},
\ee
while the collection $\UU^{(j)}$ will consist of cubes $Q=Q_1\times Q_2$ of sidelength $L_i$ for some $1\leq i\leq j$ such that $B_1$ and $B_2$ are $(\eta,L_{i+1})$-uniformly distributed on $Q_1$ and $Q_2$ respectively. Note that the cubes in $\UU^{(j)}$ may have different sizes.
The remaining collection  $\NN^{(j)}$ will consist of those cubes $Q$ of sidelength $L_j$ which are not $(\eta,L_{j+1})$-uniformly distributed.
We will stop the procedure when the total measure of the non-uniform cubes is small enough, specifically when
\be\label{3.4.3}
\sum_{Q\in \NN^{(j)}} |Q|\leq \frac{\eta}{2}
\ee
and note that such a partition satisfies the conclusions of Theorem \ref{reglemma}.

If  $[0,1]^{d_1}\times [0,1]^{d_2}\in\UU^{(0)}$,  then the sets $B_1,\,B_2$ are both $(\eps, L_1)$-uniformly distributed and Theorem \ref{reglemma} holds. We thus assume that for some $j\geq 0$ we have a partition $\PP^{(j)}$ for which \eqref{3.4.3} does not hold and let
 $Q=Q_1\times Q_2$ denote an arbitrary  cube in $\NN^{(j)}$. By our assumption both cubes have sidelength $L_j$ and $B_i$ is not $(\eta,L_{j+1})$-uniformly distributed on $Q_i$ for either $i=1$ or $i=2$.

 We assume, without loss of generality, that $i=1$.
Averaging show that for $Q_1=t_1+[0,L_j]^{d_1}$ and $L:=L_{j+1}$, we have
\be\label{3.4.5}
|E_\eta|\geq \frac{\eta^2}{2} |Q_1|
\ee
where
\be
E_\eta :=\left\{t\in Q_1\,:\, \left|\frac{|B_1\cap(t+Q_L)|}{|Q_L|}-\frac{|B_1\cap Q_1|}{|Q_1|}\right| \geq \frac{\eta}{2}\right\}.
\ee

Let $m=\lfloor L_j/L_{j+1}\rfloor$  and partition the cube $Q_1'=t_1+[0,(m+1)L]^{d_1}\supseteq Q_1$ into grids of the form $G(s_1)=s_1+\{0,L,\ldots ,mL\}^{d_1}$ with $s_1$ running through the cube $t_1+[0,L]^{d_1}$. Since $L<2^{-6}L_j$, by \eqref{3.4.5} there exist $s_1\in Q'_1$ such that
\be\label{3.4.51}
\frac{|G(s_1)\cap E_\eta|}{m^{d_1}}\geq \frac{\eta^2}{4}.
\ee

Fix such an $s_1$ and consider the partition of $Q_1$ into cubes of size $L=L_{j+1}$ and possibly rectangles, defined by the grid $G(s_1)$. Repeat the same partition of the cube $Q_2$ corresponding to a point $s_2$ which we can choose arbitrarily from a cube $Q'_2\subseteq Q_2$ of size $L$. Taking the direct product of these partitions gives a partition of the cube $Q=Q_1\times Q_2$ into cubes of size $L=L_{j+1}$ and possibly also into some $(d_1\times d_2)$-dimensional rectangles. After performing  this partition of all cubes in $\NN^{(j)}$ we obtain the new partition $\PP^{(j+1)}$ of $[0,1]^{d_1}\times [0,1]^{d_2}$.
The new cubes obtained are then partitioned into classes  $\UU^{(j+1)}$ and $\NN^{(j+1)}$ according to whether they are $(\eta,L_{j+2})$-uniform. Note that the cubes in $\UU^{(j)}$ and rectangles in $\RR^{(j)}$ remain cells of $\PP^{(j+1)}$.
Note that for each cube $Q\in\NN^{(j)}$ the total measure of all the rectangles obtained is at most $16L_{j+1}L_j^{-1}|Q|$, hence summing over all cubes the total measure of the rectangles obtained this way is at most $4L_{j+1}L_j^{-1}$. We adjoin these rectangles to $\RR^{(j)}$ to form $\RR^{(j+1)}$. Note that this way the total measure of the rectangles is always bounded by
\[\sum_{j=0}^\infty \frac{16L_{j+1}}{L_j} \leq \sum_{j=0}^\infty 2^{-(j+2)} \eta \leq \frac{\eta}{2},\]
hence \eqref{3.4.2} holds.

A key notion in regularization arguments is that of the \emph{index} or \emph{energy} of a set with respect to a partition. In our context we define it as follows.
Let $\{C_k\}_{k=1}^K$ denote the collection of cells that constitute $\PP^{(j)}$. For any given cell $C^k=Q^k_1\times Q^k_2$ in $\PP^{(j)}$, where $Q^k_i$ could be either a square or a rectangle, we let $\de^k_i$ %=\dfrac{|B\cap C_k|}{|C_k|}$
denote the relative density of $B_i$ in $Q^k_i$ for $i=1,2$, and define the \emph{energy} of $(B_1,B_2)$ with respect to  $\PP^{(j)}$ by
\be\label{3.4.6}
\EE(B_1,B_2;\PP^{(j)}):= \frac{1}{2}\,\sum_{C^k\in \PP^{(j)}} \bigl((\de^k_1)^2 + (\de^k_2)^2\bigr)\,|C^k|.
\ee

It is not hard to see that the energy is always at most 1 and is increasing when the partition is refined.
To be more precise, we say a partition $\PP'$ is a refinement of $\PP$ if every cell $C=Q_1\times Q_2$ of $\PP$ is decomposed into cells $C^{\ell,{\ell'}}=Q^\ell_1\times Q^{\ell'}_2$ of $\PP'$ so that cubes (or rectangles) $Q^\ell_1$ and $Q^{\ell'}_2$ form a partition of $Q_1$ and  $Q_2$ respectively.
Then
$|Q_1|=\sum_\ell |Q^\ell_1|$ and $|B_1\cap Q_1|=\sum_\ell |B_1\cap Q^\ell_1|$, hence writing $\de_1$ for the relative density of $B_1$ on $Q_1$ and
$\de^\ell_1$ for the relative density of $B_1$ on $Q^\ell_1$ one has
\be\label{3.4.7}
\sum_{\ell} (\de^\ell_1)^2\,|Q^\ell_1|=(\de_1)^2\,|Q_1|+\sum_\ell (\de^\ell_1-\de_1)^2\,|Q^\ell_1|.
\ee
Similarly
\be\label{3.4.8}
\sum_{\ell'} (\de^{\ell'}_2)^2\,|Q^l_2|=(\de_2)^2\,|Q_2|+\sum_{\ell'} (\de^{\ell'}_2-\de_2)^2\,|Q^{\ell'}_2|.
\ee
Multiplying equations \eqref{3.4.7} by $|Q_2|$, \eqref{3.4.8} by $|Q_1|$, and adding, we get
\be\label{3.4.9}
\sum_{\ell,{\ell'}} \bigl((\de^\ell_1)^2+(\de^{\ell'}_2)^2\bigr)\,|C^{\ell,{\ell'}}| = \bigl((\de_1)^2+(\de_2)^2\bigr)\,|C|+\sum_{\ell,{\ell'}} \bigl((\de^\ell_1-\de_1)^2+(\de^{\ell'}_2-\de_2)^2\bigr)\,|C^{\ell,{\ell'}}| .
\ee

Going back to our construction we have decomposed each cell $C^k=Q_1\times Q_2\in\NN^{(j)}$ into cubes of the form $C^{\ell,{\ell'}}=Q^\ell_1\times Q^{\ell'}_2$ where $Q^\ell_1=s_1+\ell L+Q_L$, $Q^{\ell'}_2=s_2+{\ell'}L+Q_L$ for some ${\ell}\in \{1,\ldots,m\}^{d_1}$ and ${\ell'}\in \{1,\ldots,m\}^{d_2}$, and into a collection of $(d_1+d_2)$-dimensional rectangles of small total measure. By \eqref{3.4.51} there at least $\eta^2m^{d_1}/4$ values of $\ell$ for which $|\de^\ell_1-\de_1|^2\geq \eta^2/4$.  Thus, as $|Q_1|=L_j^{d_1}$, $|Q^\ell_1|=L^{d_1}_{j+1}$ for all $\ell$, and $m=\lfloor L_j/L_{j+1}\rfloor\geq \frac{1}{2}L_j/L_{j+1}$, we have that
\be\label{3.4.10}
\sum_{\ell\in \{1,\ldots,m\}^{d_1}} (\de^\ell_1-\de_1)^2\,|Q^\ell_1| \geq \frac{\eta^4}{64}\,|Q_1|.
\ee

By \eqref{3.4.9} this implies that the energy of $(B_1,B_2)$ with respect to the collection of cells of $\PP^{(j+1)}$ contained in $C^k=Q_1\times Q_2$ given by the left side of \eqref{3.4.9} is at least
\be\label{3.4.11}
\EE(B_1,B_2;\PP^{(j+1)}|_{C^k}) \,\geq\, \frac{1}{2}\,(\de_1^2+\de_2^2)\,|C^k|+\frac{\eta^4}{128}\,|C^k|.
\ee
This holds for all non-uniform cells $C^k\in\NN^{(j)}$ and by our assumption that the total measure of $\NN^{(j)}\geq \eta/2$ it follows that
\be\label{3.4.12}
\EE(B_1,B_2;\PP^{(j+1)}) \,\geq\, \EE(B_1,B_2;\PP^{(j)}) \,+\,\ \frac{\eta^5}{256}.
\ee

Thus the procedure must stop in $j\leq 256\,\eta^{-5}$ steps providing a satisfactory partition. As explained above this leads to a cell $C=Q_1\times Q_2$ satisfying the conclusions of Theorem \ref{reglemma}.\qed
%\end{proof}

\bigskip

\bigskip

%\newpage

\bigskip
\noindent
\emph{Acknowledgements.} We would like to thank the anonymous referee for useful comments and suggestions which have greatly improved the exposition of this article.

\end{document}